\documentclass{article}

\usepackage{booktabs, multirow, bm}
\usepackage{amsmath, amssymb, amsthm, mathrsfs}
\usepackage{indentfirst}
\usepackage{graphicx, tikz}
\usepackage{geometry}
\usepackage{caption}
\usepackage{subfigure, booktabs}

\usepackage{hyperref}

\usepackage{cleveref}
\crefname{section}{Section}{Sections}
\crefname{figure}{Figure}{Figures}
\crefname{table}{Table}{Tables}
\crefname{equation}{}{}
\crefname{theorem}{Theorem}{Theorems}
\crefname{lemma}{Lemma}{Lemmas}
\crefname{remark}{Remark}{Remarks}
\crefname{problem}{Problem}{Subproblems}

\newtheorem{theorem}{Theorem}[section]
\newtheorem{subproblem}{Subproblem}[section]

\newtheorem{lemma}{Lemma}[section]

\theoremstyle{definition}
\newtheorem{example}{\noindent Example}

\graphicspath{{figure/}}

\begin{document}
	
	\title{Jointly determining the point sources and obstacle from Cauchy data}
	
	\author{ Deyue Zhang\thanks{School of Mathematics, Jilin University, Changchun, China, {\it dyzhang@jlu.edu.cn}},
	              Yan Chang\thanks{School of Mathematics, Harbin Institute of Technology, Harbin, China. {\it 21B312002@stu.hit.edu.cn}}
              \ and Yukun Guo\thanks{School of Mathematics, Harbin Institute of Technology, Harbin, China. {\it ykguo@hit.edu.cn} (Corresponding author)},}
              
\date{}
\maketitle
	
\begin{abstract}
A numerical method is developed for recovering both the source locations and the obstacle from the scattered Cauchy data of the time-harmonic acoustic field. First of all, the incident and scattered components are decomposed from the coupled Cauchy data by the representation of the single-layer potentials and the solution to the resulting linear integral system. As a consequence of this decomposition, the original problem of joint inversion is reformulated into two decoupled subproblems: an inverse source problem and an inverse obstacle scattering problem. Then, two sampling-type schemes are proposed to recover the shape of the obstacle and the source locations, respectively. The sampling methods rely on the specific indicator functions defined on target-oriented probing domains of circular shape. The error estimates of the decoupling procedure are established and the asymptotic behaviors of the indicator functions are analyzed. Extensive numerical experiments are also conducted to verify the performance of the sampling schemes.
\end{abstract}
	
	
\maketitle	
	
\section{Introduction}

The inverse scattering problems concerning the reconstruction of unknown targets from measured scattering data have received enduring attention due to their wide applications in medical imaging, nondestructive testing, ocean acoustics, and radar sensing (see, e.g., \cite{CK19}). Roughly speaking, the conventional inverse scattering problems can broadly be divided into two modalities: inverse obstacle/medium scattering problems and inverse source problems. For inverse obstacle scattering models, the excitation source is known and, in general, can be artificially deployed as desired. Then the interaction of this active incidence with the unknown impenetrable scatterer induces an outgoing response, which is called the scattered wave. By measuring such data, the goal of the inverse obstacle problem is to reversely identify the features of the obstacle. On the contrary, if there is no perturbation on the incoming wave, then the signal caused by the source emanation/radiation can be directly recorded by the receivers. Correspondingly, the backward process of determining the source's fluctuation or distribution is acknowledged as the inverse source problem.

With the rapid growth of the inverse scattering theory,  a number of powerful algorithms have been developed to tackle the inverse obstacle problems or inverse source problems. State-of-the-art approaches such as the sampling method, the decomposition method, the iteration scheme, and the recursive linearization method, to name a few, have been successfully applied to various inverse scattering problems (see \cite{AKS09, Bao1, Cakoni1, CCH13, LL17, MHC14, ZWWG22} and the reference therein).  In addition, concerning the inverse source problem in domains with possible unknown obstacles, we refer to \cite{LTT11} for a source discovery algorithm in the high-frequency setting. We are now concerned with a practical application where the autonomous observer sent into an unknown environment tries to efficiently determine the signal sources and the non-penetrable solid obstacles located along the observer's path. As an intrinsic couple of the inverse obstacle and inverse source problems, such co-inversion problems attract great interest and we refer to \cite{CG22, FDTZ20, LL17, LHY2, ZWG22, ZWGC23} for the relevant studies on simultaneous recovery of the multiple targets in the scattering problems. Compared with the inverse problem with a single target to be determined, the joint inversion problem inherits the ill-posedness and nonlinearity from the inverse obstacle problem. In addition, the lack of measured information further makes this problem more challenging.

%Over recent years, intensive attention has been devoted to investigating inverse scattering problems and inverse source problems. 

%Concerning the inverse source problem, numerical methods for determining the source points have drawn a lot of interest over the last few years. In \cite{ZGLL19}, the authors develop simple and effective sampling schemes to locate multiple multipolar sources from boundary measurements for the Helmholtz equation. In \cite{BGWL20}, the authors develop several indicator functions to identify the point sources for the Helmholtz equation from far-field data. In \cite{AKS09}, an iterative method based on the reciprocity gap principle is proposed to detect the source outside some known obstacle. 

In this paper, we propose a novel sampling-type method to jointly determine the obstacle and its excitation source from Cauchy data. Using near-field Cauchy measurement (both Dirichlet data and Neumann data) is beneficial to the solution of some inverse scattering problems. For example, the Cauchy data avoids the numerical evaluation of the desired normal derivative on the measurement surface and can be naturally integrated into the reciprocity gap functional in imaging schemes that are based on integral equations. In \cite{CH05} and \cite{SGM16}, the Cauchy measurements are utilized to treat the inverse scattering problem of imaging the obstacle or cavity. We also refer to \cite{ZGLL19} for direct locating acoustic multipolar sources and \cite{WGB22} for recovering elastic moment tensor point sources using Cauchy data. In particular, the Cauchy data in the present paper has the advantage of reducing the twinned measurement surfaces \cite{ZWGC23} to a single one and it enables us to identify both the interior and exterior source points. Another advantage of the proposed method is the easy-to-implement decomposition scheme, which can be manipulated via the regularized linear integral equations. In addition, the subsequent sampling procedure is adaptively designed such that the indicators and the circular probing domains fit well the goal of reconstruction. Moreover, the overall algorithm does not rely on any solution process of the forward scattering problem, therefore the cost of computation is low. All these aforementioned features constitute the novelties of this work.

The rest of this paper is arranged as follows. We first present the mathematical formulation of the inverse problem and establish a result on the unique identification of the point sources in \cref{sec: model_problem}. In \cref{sec: decoupling} we introduce the layer-potential approach of decomposing the co-inversion problem. The stability of the decomposition is then analyzed. In \cref{sec:algorithms}, we propose several sampling schemes for locating the source points and imaging the shape of the obstacle. The indicating behavior of the imaging functions is also investigated. Finally, we present some numerical experiments that demonstrate the performance of the method in \cref{sec: numerical_experiments}.
	
%%%%%%%%%%%%%%%%%%%%%%%%%%%%%%%%%%%%%%%%%%%%%%%%%%%%%%%%%%%%%%%%%%%%%%%
\section{Problem setting and uniqueness}\label{sec: model_problem}
%%%%%%%%%%%%%%%%%%%%%%%%%%%%%%%%%%%%%%%%%%%%%%%%%%%%%%%%%%%%%%%%%%%%%%%%
	
Let us now introduce the mathematical model of the forward and inverse problem. In this paper, we restrict ourselves to the two-dimensional case and remark that the extension to the three-dimensional case follows analogously. Let $D \subset\mathbb{R}^2$ be an open and simply connected domain with $C^2$ boundary $\partial D$. For a generic point $z\in \mathbb{R}^2\backslash\overline{D}$, the incident field $u^i$ due to the point source located at $z$ is given by
\begin{equation}\label{pointsource}
		u^i (x; z)=\Phi(x, z):=\frac{\mathrm{i}}{4}H_0^{(1)}(k|x-z|), \quad x\in  \mathbb{R}^2\setminus \left({\overline{D}\cup \{z\}}\right),
\end{equation}
where $H_0^{(1)}$ is the Hankel function of the first kind of order zero, and $k>0$ is the wavenumber. Then, the forward scattering problem can be formulated as the following: find the scattered field $u^s(x; z)$ which satisfies the following boundary value problem:
\begin{align}
	\Delta u^s+ k^2 u^s & = 0\quad \mathrm{in}\ \mathbb{R}^2\backslash\overline{D},\label{eq:Helmholtz} \\
	\mathscr{B}u & = 0 \quad \mathrm{on}\ \partial D, \label{eq:boundary_condition} \\
	\lim\limits_{r:=|x|\to\infty} \sqrt{r}\bigg(&\frac{\partial u^s}{\partial r} - \mathrm{i} k u^s\bigg)=0, \label{eq:Sommerfeld}
\end{align}
where $u(x; z)=u^i(x; z)+u^s(x; z)$ denotes the total field and \cref{eq:Sommerfeld} is the Sommerfeld radiation condition. The boundary operator $\mathscr{B}$ in \cref{eq:boundary_condition} is defined by
\begin{equation}\label{BC}
	\mathscr{B}u=
	\begin{cases}
		u, & \text{for a sound-soft obstacle},  \\
		\partial_\nu u+ \mathrm{i} k\lambda u, & \text{for an impedance obstacle},
	\end{cases}
\end{equation}
where $\nu$ is the unit outward normal to $\partial D$ and $\lambda$ is a real parameter. This general boundary condition \cref{BC} covers the Dirichlet/sound-soft boundary condition, the Neumann/sound-hard boundary condition ($\lambda=0$), and the impedance boundary condition ($\lambda\neq 0$). It is well known that the forward scattering problem \eqref{eq:Helmholtz}-\eqref{eq:Sommerfeld} admits a unique solution $u^s\in H_{\rm loc}^1(\mathbb{R}^2\backslash\overline{D})$ (see, e.g., \cite{CK19}).
	
Let $S=\cup_{j=1}^{N}\{z_j\}\subset \mathbb{R}^2\backslash \overline{D}$ be a set of distinct source points with $N$ the number of source points, and 
$$
	u^i(x; S)=\sum\limits_{j=1}^{N} u^i(x; z_j).
$$  

Denote by $u^s(x; S)$ the scattered field corresponding to the incident field $u^i(x; S)$, and let  $u(x; S)=u^i(x; S)+u^s(x; S)$ be the total field. Take a smooth measurement curve $\Gamma=\partial \Omega$  such that $D\subset\Omega$ and $\Gamma\cap S=\emptyset$. Collect the Cauchy data $u(x; S)$ and $\partial_{n} u(x; S)$ on the curve $\Gamma$ with  $n$ being the unit outward normal to $\Gamma$. Then, the co-inversion problem under consideration is to determine the obstacle-source pair $(\partial D, S)$ from the measurements $\mathbb{U}:=\{u(x; S), \partial_{n} u(x; S): x\in \Gamma \}$, namely,
	\begin{equation}\label{co-inversion}
		\mathbb{U}\rightarrow (\partial D, S).
	\end{equation}
	
	\begin{figure}[htp]
		\centering
		\begin{tikzpicture}[scale=1.2, thick]
			\pgfmathsetseed{8}
			\draw plot [smooth cycle, samples=5, domain={1:8}]
			(\x*360/8+5*rnd:.7cm+1.cm*rnd) node at (0.1,0.1) {$D$};
			\draw node at (-0.75, 0.9) {$\partial D$};
			
			%\draw[magenta] (0.1,0) circle (1.7cm) node at (-1.4,1.2){$ \Gamma_1 $};
			\draw[magenta]  node at (2.1,-1.2) {$ \Omega$};
			\draw[magenta] (0.1, 0) circle (2.5cm) node at (-1.5,1.6) {$ \Gamma $};
			%\draw[dashed]  node at (1.6,-1.2) {$\Omega_2 $};
			
			\filldraw [red] (-2.6,0.2) circle (1.2pt);
			\filldraw [red] (-2.55,-0.4) circle (1.2pt);
			\filldraw [red] (-2.65,-0.1) circle (1.2pt);
			\filldraw [red] (-1.9,-0.3) circle (1.2pt);
			\filldraw [red] (-1.95,0.0) circle (1.2pt);
			%\filldraw [red] (-2.15,0.3) circle (1.2pt);
				
			\filldraw [red] (0.2,1.5) circle (1.2pt);
			\filldraw [red] (-0.1,1.45) circle (1.2pt);	
			\filldraw [red] (0.5,1.45) circle (1.2pt);
			\filldraw [red] (0.2,2.7) circle (1.2pt);
			\filldraw [red] (-0.1,2.7) circle (1.2pt);
			
			\filldraw [red] (1,-1.8) circle (1.2pt);
			\filldraw [red] (1.25,-1.7) circle (1.2pt);
			\filldraw [red] (1.5,-1.55) circle (1.2pt);
				
			\filldraw [red] (-0.5,-2.65) circle (1.2pt);
			\filldraw [red] (-0.9,-2.6) circle (1.2pt);
			\filldraw [red] (-1.2,-2.45) circle (1.2pt);	
		
			\filldraw [red] (2.8,0.5) circle (1.2pt);
			\filldraw [red] (2.8,-0.1) circle (1.2pt);
			\filldraw [red] (2.9,0.2) circle (1.2pt);
		%	\draw[dashed] [blue] (0.1, 0) circle (0.6cm) node at (0.3, 0.) {$B_1$};
		%	\draw[dashed] [blue] (0.1, 0) circle (3.3cm) node at (2.4, -1.9) {$B_2$};
			
			%\draw[dashed] [blue] (0,0) circle (4cm) node at (0,-3.5) {$B_2$};
		\end{tikzpicture}
		\caption{Illustration of the co-inversion for imaging the obstacle and source points.} \label{fig:illustration1}
	\end{figure}
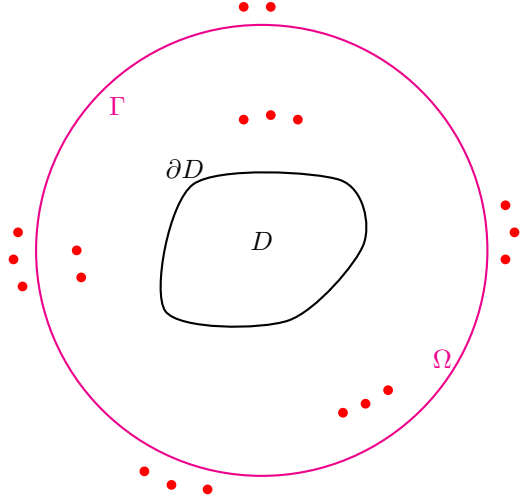
	
	We refer to Figure \ref{fig:illustration1} for an illustration of the geometry setup of the co-inversion problem \eqref{co-inversion}. In \cref{fig:illustration1}, the source points are marked by the red dots and the measurement curve is denoted by the magenta circle. We first present a uniqueness result on identifying the sources.
	
\begin{theorem}\label{Thm2.1}%-------------------------------------------------------------------------
		The locations of the source points $S$ can be uniquely determined by the Cauchy data $\mathbb{U}$. 
\end{theorem}
\begin{proof}
	Assume $D_1$ and $D_2$ are two obstacles such that $ D_\ell\subset \Omega, \ell=1,2$. Let $\mathcal{S}_1$ and $\mathcal{S}_2$ be two sets of distinct source points, and denote by $u(x;\mathcal{S}_1)$ and $u(x; \mathcal{S}_2)$ the total fields generated by $D_1$, $u^i(x; \mathcal{S}_1)$ and $D_2$, $u^i(x; \mathcal{S}_2)$, respectively. Assume that
	\begin{equation*}
		u(x; \mathcal{S}_1)=u(x; \mathcal{S}_2), \quad \partial_{n} u(x; \mathcal{S}_1)=\partial_{n} u(x; \mathcal{S}_2), \quad x\in \Gamma.
	\end{equation*}
		We claim that $\mathcal{S}_1=S_2$. Otherwise, let $w_1\in \mathcal{S}_1$, $w_2\in \mathcal{S}_2$ and $w_1\ne w_2 $. From the Holmgren’s principle \cite[Theorem 2.3]{CK19}, we have
		\begin{equation*}
			u(x; \mathcal{S}_1)=u(x; \mathcal{S}_2),\quad \mathrm{in}\ B_2\setminus ( D\cup \{S_1,S_2\}).
		\end{equation*}
		By letting $x\to w_2$ and using the boundedness of $u^s(x; \mathcal{S}_\ell) (\ell=1,2)$, we have that $u(x; \mathcal{S}_1)$ is bounded and $u(x; \mathcal{S}_2)$ tends to infinity, which is a contradiction. Hence $\mathcal{S}_1=\mathcal{S}_2$. The proof is completed.
		%This implies that the Cauchy data $\{u(x; S),\partial_{n}	u(x; S): x\in \Gamma\}$ uniquely determine the location of the source point $z_j, j=1,\cdots, N$, i.e., $S$ can be uniquely determined by $\mathbb U$.	
	\end{proof}
	
%%%%%%%%%%%%%%%%%%%%%%%%%%%%%%%%%%%%%%%%%%%%%%%%%%%%%%%%%%%%%%%%%%%%%
\section{Decoupling the co-inversion problem}\label{sec: decoupling}
%%%%%%%%%%%%%%%%%%%%%%%%%%%%%%%%%%%%%%%%%%%%%%%%%%%%%%%%%%%%%%%%%%%%%
	
Based on the single-layer representation, we propose a method to decompose the total field into two parts. Then, the co-inversion problem can be decoupled into two subproblems: an inverse source problem and an inverse obstacle scattering problem. 
	
More precisely, the set $S$ is divided into $S_1$ and $S_2$ by the measurement curve $\Gamma$, that is,
	\begin{equation*}
		S=S_1\cup S_2, \quad  S_1\subset \Omega\setminus \overline{D},
		\quad S_2\subset \mathbb{R}^2 \setminus \overline{\Omega}.
	\end{equation*}
	Let
	\begin{equation*}
		u^i(x;S_\ell)=\sum\limits_{z_j\in S_\ell}u^i(x; z_j), \quad \ell=1,2.
	\end{equation*}
	Then, it is readily seen 
	\begin{equation*}
		u(x;S)=v(x;S)+u^i(x;S_2),
	\end{equation*}
	where $v(x;S)=u^i(x;S_1)+u^s(x;S)$. In the following, we will decouple  $v(x; S)$ and $u^i(x; S_2)$ on $\Gamma$ from the total field $u(x; S)$ by the single-layer potential method.
	
To this aim, we take two closed curves $\Gamma_1=\partial\Omega_1$ and  $\Gamma_2=\partial\Omega_2$ as schematically shown in Figure \ref{fig:illustration2} such that $(D\cup S_1)\subset \Omega_1\subset \Omega$, $(\Omega\cup S_2)\subset \Omega_2$, and $k^2$ is not a Dirichlet eigenvalue for the negative Laplacian in $\Omega_1$ or $\Omega_2$.
	
	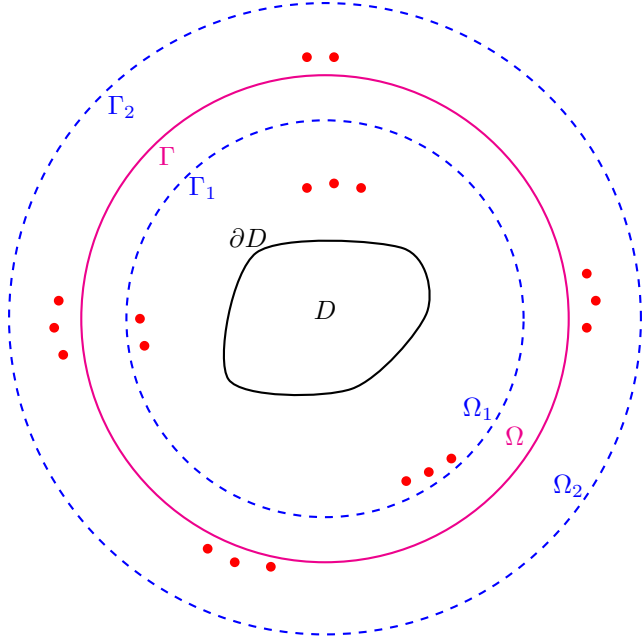
\begin{figure}
		\centering
		\begin{tikzpicture}[scale=1.2, thick]
		\pgfmathsetseed{8}
		\draw plot [smooth cycle, samples=5, domain={1:8}]
		(\x*360/8+5*rnd:.7cm+1.cm*rnd) node at (0.1,0.1) {$D$};
		\draw node at (-0.75, 0.9) {$\partial D$};
		
		%\draw[magenta] (0.1,0) circle (1.7cm) node at (-1.4,1.2){$ \Gamma_1 $};
		\draw[magenta]  node at (2.2,-1.3) {$ \Omega$};
		\draw[magenta] (0.1, 0) circle (2.7cm) node at (-1.65,1.8) {$ \Gamma $};
		%\draw[dashed]  node at (1.6,-1.2) {$\Omega_2 $};
		
		\filldraw [red] (-2.85,0.2) circle (1.2pt);
		\filldraw [red] (-2.8,-0.4) circle (1.2pt);
		\filldraw [red] (-2.9,-0.1) circle (1.2pt);
		\filldraw [red] (-1.9,-0.3) circle (1.2pt);
		\filldraw [red] (-1.95,0.0) circle (1.2pt);
		%\filldraw [red] (-2.15,0.3) circle (1.2pt);
		
		\filldraw [red] (0.2,1.5) circle (1.2pt);
		\filldraw [red] (-0.1,1.45) circle (1.2pt);	
		\filldraw [red] (0.5,1.45) circle (1.2pt);
		\filldraw [red] (0.2,2.9) circle (1.2pt);
		\filldraw [red] (-0.1,2.9) circle (1.2pt);
		
		\filldraw [red] (1,-1.8) circle (1.2pt);
		\filldraw [red] (1.25,-1.7) circle (1.2pt);
		\filldraw [red] (1.5,-1.55) circle (1.2pt);
		
		\filldraw [red] (-0.5,-2.75) circle (1.2pt);
		\filldraw [red] (-0.9,-2.7) circle (1.2pt);
		\filldraw [red] (-1.2,-2.55) circle (1.2pt);	
		
		\filldraw [red] (3,0.5) circle (1.2pt);
		\filldraw [red] (3,-0.1) circle (1.2pt);
		\filldraw [red] (3.1,0.2) circle (1.2pt);
		\draw[dashed] [blue] (0.1, 0) circle (2.2cm) node at(1.8,-1.)  {$\Omega_1$} node at (-1.25,1.45) {$ \Gamma_1 $};
		\draw[dashed] [blue] (0.1, 0) circle (3.5cm) node at (2.8, -1.85) {$\Omega_2$} node at (-2.15,2.35) {$ \Gamma_2 $};
		
		%\draw[dashed] [blue] (0,0) circle (4cm) node at (0,-3.5) {$B_2$};
		\end{tikzpicture}
		\caption{Geometry setting of the decomposition strategy.} \label{fig:illustration2}
	\end{figure}

%%%%%%%%%%%%%%%%%%%%%%%%%%%%%%%%%%%%%%%%%%%%%%%%%%%%%%%%%%%%%%%%%%%%%
\subsection{Decomposition of the total field}\label{sec:Decomposition}
%%%%%%%%%%%%%%%%%%%%%%%%%%%%%%%%%%%%%%%%%%%%%%%%%%%%%%%%%%%%%%%%%%%%%
	
Motivated by the recent work \cite{ZWGC23} for solving the interior co-inversion problem, we approximate the field $v(x; S)$ and the incident field $u^i(x; S_2)$ by the following single-layer potentials, respectively,
\begin{align}
	& v(x; S)  \approx \int_{\Gamma_1}\Phi(x, y)\varphi_{1}(y)\mathrm{d}s(y),
	\quad x\in \mathbb R^2\backslash\overline \Omega_1, \label{singlelayer1} \\
	& u^i(x; S_2)  \approx \int_{ \Gamma_2}\Phi(x, y)\varphi_{2}(y)\mathrm{d}s(y),
	\quad x\in \overline \Omega, \label{singlelayer2}
\end{align}
where $\varphi_{1}\in L^2(\Gamma_1),\ \varphi_{2}\in L^2(\Gamma_2)$ are unknown density functions.
	
	To determine the density functions, we introduce the operators $\mathcal{S}_{m}, \mathcal{K}_{m}: L^2(\Gamma_m)\to L^2(\Gamma), m=1,2$,
\begin{align*}
&	(\mathcal{S}_{m}\psi_m)(x):=\int_{\Gamma_m}\Phi(x, y)\psi_m(y)\mathrm{d}s(y),\quad x\in \Gamma,\\
&	(\mathcal{K}_{m}\psi_m)(x):=\frac{\partial}{\partial n(x)}\int_{\Gamma_m}\Phi(x, y)\psi_m(y)\mathrm{d}s(y),\quad x\in \Gamma,
\end{align*}
where $\psi_m\in L^2(\Gamma_m)$, $m=1,2$.
	
	Now, by using the total fields $u_\Gamma=u(x; S)$, $\partial_nu_\Gamma=\partial_{n} u(x; S)$, $x\in\Gamma$, we derive the following equations for $\varphi_1$ and $\varphi_2$,
\begin{align*}
		\mathcal{S}_{1}\varphi_{1}+\mathcal{S}_{2}\varphi_{2} & =u_\Gamma,\quad \mathrm{on}\ \Gamma, \\
		\mathcal{K}_{1}\varphi_{1}+\mathcal{K}_{2}\varphi_{2} & =\partial_n u_\Gamma,\quad\mathrm{on}\ \Gamma,
\end{align*}
or equivalently,
\begin{equation}\label{BIE}
		\mathcal{T}\boldsymbol{\varphi}=\boldsymbol{u},
\end{equation}
where $\boldsymbol{\varphi}=(\varphi_1,\varphi_2)^\top $, $\boldsymbol{u}=(u_\Gamma, \partial_n u_\Gamma)^\top$, the  operator $\mathcal{T}: L^2(\Gamma_1)\times L^2(\Gamma_2)\to L^2(\Gamma)\times L^2(\Gamma)$ is defined by
\begin{equation}\label{operatorS}
		\mathcal{T} :=
		\begin{bmatrix}
			\mathcal{S}_{1} &  \mathcal{S}_{2} \\
			\mathcal{K}_{1} &  \mathcal{K}_{2}
		\end{bmatrix}.
\end{equation}

	The following theorem indicates that the operator $\mathcal{T}$ is compact, hence the operator equation \eqref{BIE} is ill-posed and should be solved by a regularization method.
	
\begin{theorem} \label{Theorem3.1}%--------------------------------------------------------------------------
	The operator $\mathcal{T}$, defined by \eqref{operatorS}, is compact, injective and has dense range.
\end{theorem}%----------------------------------------------------------------------------------------
\begin{proof}
		The operator $\mathcal{T}$ is compact due to the compactness of the operators $\mathcal{S}_{m}, \mathcal{K}_{m}$, $m=1,2$.
		
		We next consider the injectivity of $\mathcal{T}$. Let $\mathcal{T}\boldsymbol{\psi}=\boldsymbol{0}$, where $\boldsymbol{\psi}=(\psi_1,\psi_2)^\top\in L^2(\Gamma_1)\times L^2(\Gamma_2)$.
		Then
		$$
		V_m(x):=\int_{\Gamma_m}\Phi(x, y)\psi_m(y)\mathrm{d}s(y),\quad x\in\mathbb R^2\backslash \Gamma_m,\  m=1,2,
		$$
		satisfy
		\begin{align*}
				V_1+V_2  =0,\quad 
			\partial_n\left( V_1+  V_2\right)=0,\quad \mathrm{on}\  \Gamma.
		\end{align*}
	By using the Holmgren’s principle \cite[Theorem 2.3]{CK19}, we have
		$V_1+V_2=0$ in $\Omega_2\backslash \Omega_1$. Further, we have
		\begin{align*}
			\Delta(V_1+V_2)+k^2(V_1+V_2) & =0,\quad \mathrm{in}\ \Omega_1,\\
			V_1+V_2 & =0,\quad \mathrm{on}\ \Gamma_1.
		\end{align*}
		By using the assumption that $k^2$ is not a Dirichlet eigenvalue for the negative Laplacian in $B_1$, we know $V_1+V_2=0$ in $\Omega_1$, which, together with the jump relations \cite[Theorem 3.1]{CK19}, yields $\psi_1=0$. Therefore, we have $V_2=0$ in $\Omega_2$ and further the potential function $V_2$ satisfies
		\begin{align*}
			\Delta V_2+k^2V_2 & =0,\quad \mathrm{in}\ \mathbb R^2\backslash\overline \Omega_2,\\
			V_2 & =0,\quad \mathrm{on}\ \Gamma_2,
		\end{align*}
		and
		$$
		\lim_{r=|x|\to\infty}\sqrt r\left(\frac{\partial V_2}{\partial r}-\mathrm{i}kV_2\right)=0.
		$$
		And the uniqueness of the exterior scattering problem implies that $V_2=0$ in $\mathbb R^2\backslash\overline \Omega_2$. Again by using the jump relations \cite[Theorem 3.1]{CK19}, we see $\psi_2=0$. Hence the operator $\mathcal{T}$ is injective.
		
		Let $ \mathcal{T}^*\bm{g}=\bm{0}$, where $ \bm{g}=(g_1,g_2)^\top\in L^2(\Gamma)\times L^2(\Gamma)$. Introduce the following potential functions	
		\begin{align*}
		&V^*(y):=\int_{\Gamma}\Phi(x, y)\overline{g_1(x)}\,\mathrm{d}s(x),\quad y\in\mathbb R^2\backslash \Gamma,\\
		&D^*(y):=\int_{\Gamma}\frac{\partial \Phi(x, y)}{\partial n(y)}\overline{g_2(x)}\,\mathrm{d}s(x),\quad y\in\mathbb R^2\backslash \Gamma.
		\end{align*}
		Then, on one hand, it is ready to see that
		\begin{align*}
			\Delta(V^*+D^*)+k^2(V^*+D^*) & =0,\quad \mathrm{in}\ \Omega_1,\\
			V^*+D^* & =0,\quad \mathrm{on}\ \Gamma_1.
		\end{align*}
		By the assumption $k^2$ is not a Dirichlet eigenvalue for the negative Laplacian in $\Omega_1$, we see $V^*+D^*=0$ in $\Omega_1$. Then, from the analyticity of $V^*$ and $D^*$, we know $V^*+D^*=0$ in $\Omega$.
		On the other hand, we see that
		\begin{align*}
			\Delta(V^*+D^*)+k^2(V^*+D^*) & =0,\quad \mathrm{in}\ \mathbb R^2\backslash\overline \Omega_2,\\
			V^*+D^* & =0,\quad \mathrm{on}\  \Gamma_2,\\
			\lim_{r=|x|\to\infty}\sqrt r\left(\frac{\partial (V^*+D^*)}{\partial r}-\mathrm{i}k(V^*+D^*)\right) & =0.
		\end{align*}
		By the uniqueness of the solution of the exterior scattering problem, we have $V^*+D^*=0$ in $\mathbb R^2\backslash\overline \Omega_2$. Then the analyticity of $V^*$ and $D^*$  implies $V^*+D^*=0$ in $\mathbb R^2\backslash\overline \Omega$. Therefore, we derive that
		\begin{align*}
			V^*+D^* =0,\quad \mathrm{in}\  \mathbb R^2\backslash \Gamma.
		\end{align*}
		 Again, from the jump relations \cite[Theorem 3.1 ]{CK19}, we obtain $g_1=0, g_2=0$. Hence the operator $\mathcal{T}^*$ is injective and by \cite[ Theorem 4.6]{CK19} the range of $\mathcal{T}$ is dense in $L^2(\Gamma)\times L^2(\Gamma)$.
	\end{proof}
	
	Due to the ill-posedness, we need to consider the perturbed equation
	\begin{equation}\label{BIEP}
		\mathcal{T}\boldsymbol{\varphi}^\delta=\boldsymbol{u}^\delta,
	\end{equation}
	where $\boldsymbol{u}^\delta\in L^2(\Gamma)\times L^2(\Gamma)$ are measured noisy data satisfying $\| \boldsymbol{u}-\boldsymbol{u}^{\delta} \|_{L^2(\Gamma)\times L^2(\Gamma)} \le \delta$ with $0<\delta<1$.
	
	Seeking for a (Tikhonov) regularized solution to equation \eqref{BIEP} is equivalent to the problem of finding the solution to the following equation:
	\begin{equation}\label{Tikhonov}
		\alpha\bm{\varphi}^{\alpha,\delta}+\mathcal{T}^*\mathcal{T}\bm{\varphi}^{\alpha,\delta}=\mathcal{T}^*\bm{u}^\delta,
	\end{equation}
   	where the adjoint operator $\mathcal{T}^*: L^2(\Gamma)\times L^2(\Gamma)\to  L^2(\Gamma_1)\times L^2(\Gamma_2)$ is defined by
   \begin{equation*}
   	\mathcal{T^*} :=
   	\begin{bmatrix}
   		\mathcal{S}^*_{1} &  \mathcal{K}^*_{1} \\
   		\mathcal{S}^*_{2} &  \mathcal{K}^*_{2}
   	\end{bmatrix},
   \end{equation*}
   with $\mathcal{S}^*_{m}, \mathcal{K}^*_{m}: L^2(\Gamma)\to L^2(\Gamma_m), m=1,2$, given by
    \begin{align*}
    &&(\mathcal{S}^*_{m}\phi)(x):=\int_{\Gamma}\overline{\Phi(x, y)}\phi(y)\mathrm{d}s(y),\quad
    x\in \Gamma_m,\quad \phi\in L^2(\Gamma), \\
    &&(\mathcal{K}^*_{m}\psi)(x):=\int_{\Gamma}\overline{ \frac{\partial\Phi(x, y)}{\partial n(y)}}\psi(y)\mathrm{d}s(y),\quad
    x\in \Gamma_m, \quad \psi\in L^2(\Gamma). 
    \end{align*}

The regularized solution to \eqref{BIEP} is the unique minimum of the Tikhonov functional
\begin{equation*}
		J_{\alpha}(\boldsymbol{\varphi}):=\left\| \mathcal{T}\boldsymbol{\varphi}-\bm{u}^\delta \right\|^2_{L^2(\Gamma)\times L^2(\Gamma)}+\alpha\left\| \boldsymbol{\varphi}\right \|^2_{L^2(\Gamma_1)\times L^2(\Gamma_2)}.
\end{equation*}
	In this paper, the regularization parameter $\alpha=\alpha(\delta)>0$ is chosen by the Morozov's discrepancy principle, and we obtain the following regularized approximation on the field $v(x; S)$ and the incident field $u^i(x; S_2)$,
	\begin{align}
		& v_{\alpha(\delta),\delta}(x; S)  = \int_{\Gamma_1}\Phi(x, y)\varphi^{\alpha(\delta),\delta}_{1}(y)\mathrm{d}s(y),
		\quad x\in \mathbb R^2\backslash\overline \Omega_1, \label{Tikhonov_us}\\
		& u^i_{\alpha(\delta),\delta}(x; S_2)  = \int_{\Gamma_2}\Phi(x, y)\varphi^{\alpha(\delta),\delta}_{ 2}(y)\mathrm{d}s(y),
		\quad x\in \Omega_2. \label{Tikhonov_ui}
	\end{align}
	
%%%%%%%%%%%%%%%%%%%%%%%%%%%%%%%%%%%%%%%%%%%%%%%%%%%%%%%%%%%%%%%%%%%%%
\subsection{Stability of the decomposition}
%%%%%%%%%%%%%%%%%%%%%%%%%%%%%%%%%%%%%%%%%%%%%%%%%%%%%%%%%%%%%%%%%%%%%

In this subsection, we will give the error estimates of the decomposition. To this aim, we first introduce some single-layer operators $\mathcal{S}_m:L^2(\Gamma_m)\to L^2(\Sigma_m)$, $ m=1,2$,
	$$
	(\mathcal{S}_m\psi_m)(x) =\int_{\Gamma_m}\Phi(x,y)\psi_m(y)\mathrm{d}s(y),\quad x\in \Sigma_m,
	$$
	where $\Sigma_m=\partial \Omega_{\Sigma_m}$ and $\Omega_1\subseteq \Omega_{\Sigma_m} \subseteq \Omega_2$.
	
	Since $k^2$ is not a Dirichlet eigenvalue for the negative Laplacian in $B_1$, this following lemma is a direct result \cite[Theorem 5.21]{CK19}.
	\begin{lemma}\label{lemma1}
		The single-layer operator $\mathcal{S}_1$ is injective and has dense range.
	\end{lemma}

   Following the proof of  Theorem \ref{Theorem3.1}, we readily derive the following result and the proof is omitted.
	\begin{lemma}\label{lemma2}
		The single-layer operator $\mathcal{S}_2$ is injective and has dense range provided that $k^2$ is not a Dirichlet eigenvalue for the negative Laplacian in $\Omega_{\Sigma_2}$.
	\end{lemma}
	
	By using Lemma \ref{lemma1} and Lemma \ref{lemma2}, it is readily seen that for a sufficiently small positive constant $\varepsilon_0$  $(0<\varepsilon_0\ll 1)$, there exist density functions $\varphi_{1}^*\in L^2(\Gamma_1)$ and $\varphi_{2}^*\in L^2(\Gamma_2)$, such that
	\begin{align*}
		\left\|\mathcal{S}_1\varphi_{1}^*-v(\ \cdot\ ;S)\right\|_{L^2(\Sigma_1)} & <\varepsilon_0,\\
		\left\|\mathcal{S}_2\varphi_{2}^*-u^i(\ \cdot\ ;S_2)\right\|_{L^2(\Sigma_2)} & <\varepsilon_0,
	\end{align*}
where  $\Omega_1\subset \Omega_{\Sigma_1}\subset \Omega$, $\Omega\subset \Omega_{\Sigma_2}\subset B_2$, $\Omega_{\Sigma_2}\cap S_2=\emptyset$, and  $k^2$ is not a Dirichlet eigenvalue for the negative Laplacian in $\Omega_{\Sigma_\ell}$, $\ell =1,2$.
	Denote $\bm{\varphi}^*=(\varphi_{1}^*,\varphi_{2}^*)^\top$, and then, from continuous dependence of solutions on the boundary value and the regularity estimate, we obtain
	\begin{equation}\label{BIEAppro}
		\mathcal{T}\bm{\varphi}^*=\bm{u}^*,
	\end{equation}
	where
	\begin{equation}\label{varepsilon}
		\left	\|\boldsymbol{u}^*-\boldsymbol{u} \right\|_{L^2(\Gamma)\times L^2(\Gamma)}<C_0\varepsilon_0=:\varepsilon,
	\end{equation}
    with $C_0>0$.
	
	In the following, we present the main result of the error estimates.
	
	\begin{theorem}\label{Theorem3.2}
		Let $\delta+\varepsilon \le \left\| \boldsymbol{u}^{\delta} \right\|_{L^2(\Gamma)\times L^2(\Gamma)}$ with $\varepsilon$ being defined in \eqref{varepsilon}, the regularized solution $\boldsymbol{\varphi}^{\alpha(\delta),\delta} $ of \eqref{Tikhonov} satisfy $\left\|\mathcal{T}\boldsymbol{\varphi}^{\alpha(\delta),\delta}-\boldsymbol{u}^{\delta}\right\|_{L^2(\Gamma)\times L^2(\Gamma)}=\delta+\varepsilon$, $\delta\in(0,\delta_0)$. Then
		
		\noindent(a) There exists a function $\boldsymbol{h}\in L^2(\Gamma)\times L^2(\Gamma)$ such that
		$$
		\left\| \boldsymbol{\varphi}^*-\mathcal{T}^*\boldsymbol{h}\right \|_{L^2(\Gamma_1)\times L^2(\Gamma_2)}<\varepsilon ;
		$$
		\noindent(b) Let $ \| \boldsymbol{h} \|_{L^2(\Gamma)\times L^2(\Gamma)} \le E $, then
		$$
		\left\| \boldsymbol{\varphi}^{\alpha(\delta),\delta}-\boldsymbol{\varphi}^* \right\|_{L^2(\Gamma_1)\times L^2(\Gamma_2)} \le 2\varepsilon+2\sqrt{(\delta+\varepsilon)E} .
		$$
	\end{theorem}
	\begin{proof}
		\noindent(a) The range of $ \mathcal{T}^*$ is dense in $L^2(\Gamma_1)\times L^2(\Gamma_2)$, since $\mathcal{T}: L^2(\Gamma_1)\times L^2(\Gamma_2)\to L^2(\Gamma)\times L^2(\Gamma)$ is compact and injective by Theorem \ref{Theorem3.1}. Therefore, for $\varepsilon$, there exists a function $ \boldsymbol{h} \in L^2(\Gamma)\times L^2(\Gamma)$ such that $\|\boldsymbol{\varphi}^*-\mathcal{T}^*\boldsymbol{h} \|_{L^2(\Gamma_1)\times L^2(\Gamma_2)}< \varepsilon$.
		
		\noindent(b) Let $\boldsymbol{\varphi}^{\delta}:= \boldsymbol{\varphi}^{\alpha(\delta),\delta}$ be the minimum of the Tikhonov functional
		\begin{equation*}
			J^{\delta}(\boldsymbol{\varphi}):=J_{\alpha(\delta),\delta}(\boldsymbol{\varphi})=\left\| \mathcal{T}\boldsymbol{\varphi}-\boldsymbol{u}^{\delta}\right \|^2_{L^2(\Gamma)\times L^2(\Gamma)}
			+\alpha(\delta)\left\| \boldsymbol{\varphi} \right\|^2_{L^2(\Gamma_1)\times L^2(\Gamma_2)}.
		\end{equation*}
		Then, we have
		\begin{align*}
			\quad (\delta+\varepsilon)^2+\alpha(\delta)\left\| \boldsymbol{\varphi}^{\delta}\right \|^2_{L^2(\Gamma_1)\times L^2(\Gamma_2)}
			& =J^{\delta}(\boldsymbol{\varphi}^{\delta}) \le  J^{\delta}(\boldsymbol{\varphi}^*) \\
			& =\left\| \boldsymbol{u}^*-\boldsymbol{u}^{\delta} \right\|^2_{L^2(\Gamma)\times L^2(\Gamma)}
			+\alpha(\delta)\left\| \boldsymbol{\varphi}^* \right\|^2_{L^2(\Gamma_1)\times L^2(\Gamma_2)} \\
			&\le (\delta+\varepsilon)^2+\alpha(\delta)\left\| \boldsymbol{\varphi}^* \right\|^2_{L^2(\Gamma_1)\times L^2(\Gamma_2)},
		\end{align*}
		which implies $\left\|\boldsymbol{\varphi}^{\delta} \right\|_{L^2(\Gamma_1)\times L^2(\Gamma_2)} \le \left\| \boldsymbol{\varphi}^* \right\|_{L^2(\Gamma_1)\times L^2(\Gamma_2)}$ for all $\delta >0$. Hence
		\begin{align*}
			\quad \left\| \boldsymbol{\varphi}^{\delta}-\boldsymbol{\varphi}^* \right\|^2_{L^2(\Gamma_1)\times L^2(\Gamma_2)}
			&=\left\| \boldsymbol{\varphi}^{\delta} \right\|^2_{L^2(\Gamma_1)\times L^2(\Gamma_2)}-2\Re\langle\boldsymbol{\varphi}^{\delta}, \boldsymbol{\varphi}^*\rangle
			+\left\| \boldsymbol{\varphi}^* \right\|^2_{L^2(\Gamma_1)\times L^2(\Gamma_2)} \\
			&\le 2 \left( \left\| \boldsymbol{\varphi}^*\right \|^2_{L^2(\Gamma_1)\times L^2(\Gamma_2)}-\Re\langle\boldsymbol{\varphi}^{\delta}, \boldsymbol{\varphi}^*\rangle \right)\\
			&=2\Re\langle\boldsymbol{\varphi}^*-\boldsymbol{\varphi}^{\delta}, \boldsymbol{\varphi}^*\rangle .
		\end{align*}
	    where $\langle \cdot , \cdot \rangle$ denotes the $L^2$-inner product on $\Gamma_1\times\Gamma_2$.
	
		From (a), let $ \tilde{\bm{\varphi}} =\mathcal{T}^*\boldsymbol{h}_j\in L^2(\Gamma_1)\times L^2(\Gamma_2)$ such that $\left\|\tilde{\boldsymbol{\varphi}}-\boldsymbol{\varphi}^*\right\|_{L^2(\Gamma_1)\times L^2(\Gamma_2)}\le \varepsilon $. Since
		$
		\left\|\boldsymbol{u}^*-\boldsymbol{u}^\delta \right\|_{L^2(\Gamma)\times L^2(\Gamma)} \leq \left\|\boldsymbol{u}^*-\boldsymbol{u} \right\|_{L^2(\Gamma)\times L^2(\Gamma)}+\left\|\boldsymbol{u}-\boldsymbol{u}^\delta \right\|_{L^2(\Gamma)\times L^2(\Gamma)}<\varepsilon+\delta
		$, we obtain
		\begin{align*}
			\left \| \boldsymbol{\varphi}^{\delta}-\boldsymbol{\varphi}^* \right\|^2_{L^2(\Gamma_1)\times L^2(\Gamma_2)}
			&\le 2\Re\langle\boldsymbol{\varphi}^*-\boldsymbol{\varphi}^{\delta}, \boldsymbol{\varphi}^*-\tilde{\boldsymbol{\varphi}}\rangle+2\Re\langle\boldsymbol{\varphi}^*-\boldsymbol{\varphi}^{\delta}, \mathcal{T}^*\boldsymbol{h}\rangle \\
			&\le 2\varepsilon \left\| \boldsymbol{\varphi}^*-\boldsymbol{\varphi}^{\delta} \right\|_{L^2(\Gamma_1)\times L^2(\Gamma_2)}  +2\Re\langle\boldsymbol{u}^*-\mathcal{T}\boldsymbol{\varphi}^{\delta}, \boldsymbol{h}\rangle  \\
			&\le 2\varepsilon\left\| \boldsymbol{\varphi}^*-\boldsymbol{\varphi}^{\delta} \right\|_{L^2(\Gamma_1)\times L^2(\Gamma_2)} + 2\Re\langle\boldsymbol{u}^*-\boldsymbol{u}^{\delta}, \boldsymbol{h}\rangle  \\
			&\quad +2\Re\langle\boldsymbol{u}^{\delta}-\mathcal{T}\boldsymbol{\varphi}^{\delta}, \boldsymbol{h}\rangle \\
			&\le 2\varepsilon\left\| \boldsymbol{\varphi}^*-\boldsymbol{\varphi}^{\delta} \right\|_{L^2(\Gamma_1)\times L^2(\Gamma_2)}+4(\delta+\varepsilon)E.
		\end{align*}
		This means $ \left( \left\| \boldsymbol{\varphi}^*-\boldsymbol{\varphi}^{\delta} \right\|_{L^2(\Gamma_1)\times L^2(\Gamma_2)}-\varepsilon \right)^2 \le \varepsilon^2+4(\delta+\varepsilon)E $, and thus,
		$$
		\left\| \boldsymbol{\varphi}^{\delta}-\boldsymbol{\varphi}^* \right\|_{L^2(\Gamma_1)\times L^2(\Gamma_2)} \le 2\varepsilon+2\sqrt{(\delta+\varepsilon)E} .
		$$
		This completes the proof.
	\end{proof}
	
	From Theorem \ref{Theorem3.2} and $\varepsilon\ll 1$, we obtain the error estimates of the decomposition:
	\begin{align}
		\left\|v_{\alpha(\delta),\delta}(\cdot; S)-v(\cdot; S)\right\|_{L^2(\Gamma)} & \leq   C_1\sqrt{\delta+\varepsilon}, \label{errorus} \\
		\left\|u^i_{\alpha(\delta),\delta}(\cdot; S_2) -u^i(\cdot; S_2)\right\|_{L^2(\Gamma)} & \leq   C_2\sqrt{\delta+\varepsilon}, \label{errorui}
	\end{align}
	where $C_1=C_1(\Gamma)$ and $C_2=C_2(\Gamma)$ are positive constants.

%%%%%%%%%%%%%%%%%%%%%%%%%%%%%%%%%%%%%%%%%%%%%%%%%%%%%%%%%%%%%%%%%%%%%
\section{Imaging algorithms}\label{sec:algorithms}
%%%%%%%%%%%%%%%%%%%%%%%%%%%%%%%%%%%%%%
	
Once the two components have been decoded from the measurements of the total field, the co-inversion problem \eqref{co-inversion} can be  decoupled into the following two inverse problems:
	
	\begin{subproblem}\label{subPro1}
			Determine the locations of the source points $S_1$ and $S_2$ from the decoupled data  $v_{\alpha(\delta),\delta}(x;S)$ in  \eqref{Tikhonov_us} and $u^i_{\alpha(\delta),\delta}(x;S_2)$ in\eqref{Tikhonov_ui}, respectively.
	\end{subproblem}

	\begin{subproblem}\label{subPro2}
		Reconstruct the boundary of the obstacle $\partial D$ after \cref{subPro1}.
	\end{subproblem}

The aim of this section is to develop imaging algorithms to solve \cref{subPro1} and \cref{subPro2}, namely,  determine the source locations and then, find the shape of the obstacle. The stability of the numerical methods for imaging the point sources will be analyzed as well. Two direct imaging methods are developed to solve \cref{subPro1} in subsection \ref{subsec:algorithm1}.

%%%%%%%%%%%%%%%%%%%%%%%%%%%%%%%%%%%%%%%%%%%%%%%%%%%%%%%%%%%%%%%%%%%%%
\subsection{Direct imaging for locating the point sources}\label{subsec:algorithm1}
%%%%%%%%%%%%%%%%%%%%%%%%%%%%%%%%%%%%%%%%%%%%%%%%%%%%%%%%%%%%%%%%%%%%%

In this subsection, two direct sampling methods are proposed for determining the locations of the source points $S_1$ and  $S_2$, respectively.

(i). 	For any sampling point $y\in  \Omega$, we introduce the following indicator function
\begin{equation}\label{indicator1}
	I_1(y)=\frac{1}{\sqrt{R}}\Re \left\{ \left\langle v_{\alpha(\delta),\delta}(\ \cdot \ ;S),\mathrm{e}^{\mathrm{i}(k|\cdot-y|+\frac{\pi}{4})}\right\rangle_{L^2(\Gamma_R)} \right\},
\end{equation}
where $\left\langle \cdot , \cdot \right\rangle$ denotes the $L^2$-inner product, and $\Gamma_R:=\{x\in  \mathbb{R}^2:|x|=R\}$ such that $R\gg 1$. We take the maximum points of the indicator function $I_1(y)$ as the approximation of the exact source points of $S_1$.

\begin{theorem}\label{ThmIndicator}	
	For all $y\in \Omega$, we have
	\begin{align*}
		I_1(y)&= \frac{1}{2\sqrt{2\pi k R}}\sum\limits_{z_j\in S_1}\int_{\Gamma_R}\frac{\cos(k(|x-z_j|-|x-y|))}{\sqrt{|x-z_j|}}\mathrm{d}s(x) \\
		&\quad +\mathcal{O}\left(R^{-\frac{1}{2}}+d_0^{-\frac{1}{2}}\right)+\mathcal{O}\left((\delta+\varepsilon)^{\frac{1}{2}}\right),
	\end{align*}
where $d_0=\inf\limits_{z\in S_1,\, x\in\partial D}|z-x|$.
\end{theorem}
\begin{proof}
	From (3.105) in \cite{CK19}, we have the following asymptotic behavior of the Hankel functions
	$$
	H_0^{(1)}(t)=\sqrt{\frac{2}{\pi t}}\mathrm{e}^{\mathrm{i}(t-\frac{\pi}{4})}
	\left\{1+\mathcal{O}\left(\frac{1}{t}\right)\right\},\quad t\to\infty.
	$$
	Then, we deduce
	\begin{align*}
		&\quad\Re \left\{ \left\langle u^i(x; S_1),\mathrm{e}^{\mathrm{i}(k|x-y|+\frac{\pi}{4})}\right\rangle_{L^2(\Gamma_R)} \right\} \\
		&=\Re\left\{\sum\limits_{z_j\in S_1}\int_{\Gamma_R}u^i(x; z_j)\ \mathrm{e}^{-\mathrm{i}(k|x-y|+\frac{\pi}{4})}\mathrm{d}s(x) \right\} \\
		&=\Re\left\{\sum\limits_{z_j\in S_1}\int_{\Gamma_R}\frac{\mathrm{i}}{4}H_0^{(1)}(k|x-z_j|)\ \mathrm{e}^{-\mathrm{i}(k|x-y|-\frac{\pi}{4})}\mathrm{e}^{-\mathrm{i}\frac{\pi}{2}}\mathrm{d}s(x) \right\} \\
		&=\frac{1}{4}\Re\left\{\sum\limits_{z_j\in S_1}\int_{\Gamma_R}\sqrt{\frac{2}{\pi k |x-z_j|}}\ \mathrm{e}^{\mathrm{i}k(|x-z_j|-|x-y|)}\mathrm{d}s(x) \right\}+\mathcal{O}\left(R^{-\frac{1}{2}}\right) \\
		&=\frac{\sqrt{2}}{4\sqrt{\pi k}}\sum\limits_{z_j\in S_1}\int_{\Gamma_R}\frac{\cos(k(|x-z_j|-|x-y|))}{\sqrt{|x-z_j|}}\mathrm{d}s(x)+\mathcal{O}\left(R^{-\frac{1}{2}}\right).
	\end{align*}
 	Further, by \eqref{varepsilon}, Theorem \ref{Theorem3.2}, the Schwarz inequality and $|u^s(x; S_1)|=\mathcal{O}(R^{-\frac{1}{2}}d_0^{-\frac{1}{2}})$ for $x\in \Gamma_R$, we have for each $y\in B_1$,
	\begin{align*}
		I_1(y)
		&= \frac{1}{\sqrt{R}}\Re\left\{\int_{\Gamma_R}\left(v_{\alpha(\delta),\delta}(x; S)-v(x; S)\right)
		\mathrm{e}^{-\mathrm{i}\left(k|x-y|+\frac{\pi}{4}\right)}\mathrm{d}s(x) \right.\\
		&\left.  \quad+\int_{\Gamma_R}u^i(x; S_1)\mathrm{e}^{-\mathrm{i}\left(k|x-y|+\frac{\pi}{4}\right)}\mathrm{d}s(x) 
		+\int_{\Gamma_R}u^s(x; S)\mathrm{e}^{-\mathrm{i}\left(k|x-y|+\frac{\pi}{4}\right)}\mathrm{d}s(x)\right\} \\
		&=\mathcal{O}\left((\delta+\varepsilon)^{-\frac{1}{2}}\right)+\frac{\sqrt{2}}{4\sqrt{\pi k R}}\sum\limits_{z_j\in S_1}\int_{\Gamma_R}\frac{\cos(k(|x-z_j|-|x-y|))}{\sqrt{|x-z_j|}}\mathrm{d}s(x)\\
		&\quad +\mathcal{O}\left(R^{-\frac{1}{2}}+d_0^{-\frac{1}{2}}\right).
	\end{align*}
	This completes the proof.
\end{proof}

In virtue of the above theorem, we see that each function $I_1(y)$ should decay as the sampling point $y$ recedes from the corresponding source point $z_j\in S_1$. And thus the source points $S_1$ can be recovered by locating the significant local maximizers of the indicator $I_1(y)$ over the sampling region $\Omega_1$.

(ii). For any sampling point $y\in \Omega_2\backslash \Omega$, we introduce two indicator functions
\begin{align}\label{indicator2}
	I_2(y)= \int_{\Gamma}\left|u^i(x; y)- u_{\alpha(\delta),\delta}^i( x;S_2)\right|\mathrm{d}s(x),
\end{align}
and
\begin{equation}\label{indicator3}
	\widehat{I}_2(y)= {\rm Im } \left\{u_{\alpha(\delta),\delta}^i( y ;S_2)\right\}.
\end{equation}
We take the minimum points of the indicator function $I_2(y)$, or the maximum points of the indicator function $\widehat{I}_2(y)$ as the approximation of the source points of $S_2$. In the following, we will analyze some properties of the indicator functions $I_2(y)$ and $\widehat{I}_2(y)$.

First, consider the indicator function $I_2(y)$. If $S_2$ contains only one source point $z^*$, from \eqref{errorui} and the Schwarz inequality, we see 
\begin{align*}
	I_2(z^*)= \int_{\Gamma}\left| u^i (x; S_2)- u_{\alpha(\delta),\delta}^i(x  ;S_2)\right|\mathrm{d}s(x)=\mathcal{O}\left((\delta+\varepsilon)^{\frac{1}{2}}\right),
\end{align*}
which implies the indicator function $I_2$  reaches the minimum at $z^*$.
For the general case, we have
\begin{align*}
	I_2(y)&\leq\int_{\Gamma}\left| u^i (x; y)- u^i (x; S_2)\right|\mathrm{d}s(x)
		+ \int_{\Gamma}\left| u^i (x; S_2)- u_{\alpha(\delta),\delta}^i(x  ;S_2)\right|\mathrm{d}s(x)	\\
	&\leq \int_{\Gamma}\left| u^i (x; y)- u^i (x; z_l)\right|\mathrm{d}s(x)+ \int_{\Gamma}\left| \sum\limits_{z_j\in S_2, j\neq l}u^i(x; z_j)  \right|\mathrm{d}s(x)
	+\mathcal{O}\left((\delta+\varepsilon)^{\frac{1}{2}}\right).
\end{align*}
The right hand of the inequality attains the minimum at $z_l$. This, to some extent, shows the indicator function $I_2(y)$ takes minimum at $z_j\in S_2$. 

Now, we turn to the property of the indicator function $\widehat{I}_2(y)$. From \eqref{pointsource}, we have
\begin{equation*} 
	u^i (x; z)=\frac{\mathrm{i}}{4} J_0(k|x-z|)-\frac{\mathrm{1}}{4}Y_0(k|x-z|),\quad x\in  \mathbb{R}^2\backslash {\{z\}},
\end{equation*}
where $J_0$ and $Y_0$ are the Bessel functions of the first and second kind of order zero, respectively. From 
\begin{equation*} 
	u^i (x; z)-\overline{u^i (x; z)}=\frac{\mathrm{i}}{2} J_0(k|x-z|),\quad x\in  \mathbb{R}^2,
\end{equation*}
we obtain
\begin{equation} \label{DefinitionJ}
{\rm Im}\left\{	u^i (x; S_2)-\overline{u^i (x; S_2)}\right\}=\frac{1}{2} \sum\limits_{z_j\in S_2}J_0(k|x-z_j|)\triangleq \mathcal{J}(x;S_2),\quad x\in  \mathbb{R}^2.
\end{equation}
Further, by using \eqref{errorui}, we see  
	\begin{align*}
	\left\|{\rm Im}\left \{u^i_{\alpha(\delta),\delta}(\cdot; S_2)- \overline{u^i_{\alpha(\delta),\delta}(\cdot; S_2)}\right\}-\mathcal{J}(\cdot; S_2)\right\|_{L^2(\Gamma)}  \leq  C\sqrt{\delta+\varepsilon},
\end{align*}
where $C>0$ is a constant. This means ${\rm Im}\left\{u^i_{\alpha(\delta),\delta}(x; S_2)- \overline{u^i_{\alpha(\delta),\delta}(x; S_2)}\right\}$ can be an approximation of $\mathcal{J}(x;S_2)$ in $\Omega$, and then in $\Omega_2$.
Now, let $\delta>0$ be a small constant and  $B_{\delta}(z_j)=\{x\in \mathbb{R}^2 :|x-z_j|<\delta \}$. From \eqref{DefinitionJ}, it is readily seen that for $z_j\in S_2$ and $x\in B_{\delta}(z_j)$
\begin{align*} 
 \mathcal{J}(x;S_2)&=\frac{1}{2}J_0(k|x-z_j|)+\frac{1}{2} \sum\limits_{z_l\in S_2, l\neq j}J_0(k|x-z_l|)\\
 &=\frac{1}{2}J_0(k|x-z_j|)+\mathcal{O}\left((k\rho)^{-\frac{1}{2}}\right),
\end{align*}
where $\rho={\rm dist}(z_j,S_2\backslash\{z_j\})-\delta$. This shows that function $\mathcal{J}(x;S_2)$ reaches maximum $1/2$ at $z_j$, which yields that ${\rm Im}\left\{u^i_{\alpha(\delta),\delta}(x; S_2)\right\}$ takes maximum $1/4$ at $z_j$. 

Based on the aforementioned analysis, the source points $S_2$ can be recovered by locating the significant local minimizers of the indicator $I_2(y)$  or the local maximizers of the indicator $\widehat{I}_2(y)$ over the sampling region $\Omega_2\backslash \Omega$.

%%%%%%%%%%%%%%%%%%%%%%%%%%%%%%%%%%%%%%%%%%%%%%%%%%%%%%%%%%%%%%%%%%%%%
\subsection{Direct imaging of the obstacle}\label{subsec:algorithm2}
%%%%%%%%%%%%%%%%%%%%%%%%%%%%%%%%%%%%%%%%%%%%%%%%%%%%%%%%%%%%%%%%%%%%%
	
In this subsection, we consider the numerical scheme for solving \cref{subPro2}, namely, reconstructing the obstacle $D$ by the direct sampling method with the aid of the artificial source points. The reason for incorporating artificial sources is the lack of information for a good reconstruction of the obstacle because the number and distribution of the known point sources are not at our disposal.  A similar strategy for supplementing information by auxiliary source points can be found in \cite{ZWG22}. 

Let $\cup_{j=1}^M\{z_j^a\}\subset\mathbb{R}^2\backslash(D\cup S)$ be a set of auxiliary source points, which are known and additionally added to the co-inversion problem. For each $j=1,\cdots, M$, the incident field due to the point source located at $z_j^a$ is denoted by $u^i(x; z_j^a)$ and the corresponding scattered wave is given by  $u^s(x; z_j^a)$. We now introduce the new inverse scattering problem of recovering $\partial D$ from the Cauchy data $\mathbb{U}_a=\{u^s(x; z_j^a), \partial_n u^s(x; z_j^a): x\in\Gamma, j=1,\cdots, M\}$.

Let $\Omega_D\subset \Omega$ be the sampling domain for imaging the obstacle, we introduce the following imaging indicator
$$
I_D(y)=\sum_{j=1}^M\left(\left|\int_{\Gamma}u^s(x, z_j^a)\overline{\Phi(x, y)}\mathrm{d}s(x)\right|+\left|\int_{\Gamma}\partial_n u^s(x, z_j^a)\overline{\partial_n\Phi(x, y)}\mathrm{d}s(x)\right|\right)
$$
where $y\subset\Omega_D$ denotes the sampling point. It is expected that the indicator function $I_D(y)$ attains its extreme values when the sampling point approaches the true boundary $\partial D$ from inside or outside the obstacle $D$, and thus the profile of the obstacle can be qualitatively highlighted by plotting the indicator function over a suitably chosen sampling grid that covers the obstacle. Although a mathematical analysis of the indicating property is currently not available, the expected behavior is observed in all of our numerical experiments. 

In addition, to further illustrate the performance of the proposed indicator, we also recall the direct sampling method based on the imaging function
$$
I_C(y)=\sum_{j=1}^M\left|\int_{\Gamma}u^s(x, z_j^a)\overline{\Phi(x, y)}\mathrm{d}s(x)\right|,
$$
which is proposed in \cite{ItoDSM} for the inverse medium scattering problem and applicable to the inverse obstacle scattering problem, too. We shall later compare these two indicator functions in \Cref{sec: numerical_experiments}.

\section{Numerical examples}\label{sec: numerical_experiments}
%%%%%%%%%%%%%%%%%%%%%%%%%%%%%%%%%%%%%%%%

The aim of this section is to develop several numerical experiments to verify the performance of the sampling methods proposed in the previous sections. To generate the synthetic Cauchy data of the total field, the direct problem \eqref{eq:Helmholtz}--\eqref{eq:Sommerfeld} is reformulated as a boundary integral equation, which is solved by the Nystr\"{o}m method \cite{CK19}. Here, the numerical quadrature rule with $64$ equidistant grid points on $[0,2\pi]$ is adopted. The synthetic Cauchy data is collected at 512 receivers that are equidistantly distributed on the circle centered at the origin with radius of 10. To test the stability of the proposed methods, we disturb the Cauchy data $\{u(x; S),\,\partial_{n}u(x; S):x\in\Gamma\}$ with some random noise such that
\begin{align*}
u^\epsilon&:=u+\epsilon r_1|u|\mathrm{e}^{\mathrm{i}\pi r_2},\\
\partial_nu^\epsilon&:=\partial_{n}u+\epsilon r_1|\partial_nu|\mathrm{e}^{\mathrm{i}\pi r_2},
\end{align*}
where $r_1$ and $r_2$ are two uniformly distributed random numbers ranging from $-1$ to $1$ and $\epsilon>0$ is the noise level. Throughout this section, $\epsilon$ is set to be $5\%.$

In our simulations, the two auxiliary curves $\Gamma_1$ and $\Gamma_2$ are chosen to be the circles centered at the origin with radii $9$ and 18, respectively. The integrals over the two auxiliary curves $\Gamma_1$ and $\Gamma_2$ are numerically approximated by the trapezoidal rule with 512 grid points. Morozov's discrepancy principle is utilized to determine the regularization parameter $\alpha$ in \eqref{Tikhonov}.

In the direct sampling method for determining the source points, the sampling domain $\mathcal{T}=\mathcal{T}_1\cup\mathcal{T}_2$ is chosen in terms of the polar coordinate to be $\mathcal{T}_1=\{(r,\theta):r\in[3,10),\theta\in[0,2\pi)\},$ with $100\times300$ equidistantly distributed mesh grid and $\mathcal{T}_2=\{(r,\theta):r\in(10,18],\theta\in[0,2\pi)\},$ with $100\times500$ equidistantly distributed mesh grid. To recover the obstacle, the sampling domain $\mathcal{T}_3$ is chosen to be $\mathcal{T}_3=\{(r,\theta):r\in[0,3),\theta\in[0,2\pi)\},$ with $100\times500$ equidistantly distributed mesh grid.  It deserves noting that, different from the usual rectangular sampling domains as selected in most existing algorithms, in our problem setting it is more convenient to select circular sampling domains since the measurement circle could be naturally placed as the interface of the interior and exterior sampling domains for finding the point sources. To reconstruct the obstacle by the direct sampling method, the auxiliary/reference point sources are incorporated to achieve the reconstruction. %Depending on different boundary conditions (sound-soft or sound-hard),  Unless otherwise specified, 6 or 12 reference source points are attached to the co-inversion system, depending on whether the boundary condition is sound-soft or sound-hard, respectively.

In the following figures concerning the geometry setting of the problem, the black solid curves denote the boundaries of the exact obstacle. The red points indicate the exact source points. The 512 receivers are equally distributed on the blue solid lines. The green and black dashed lines denote the two auxiliary curves $\Gamma_1$ and $\Gamma_2$, respectively. The small blue `+' markers mark the auxiliary source points. In the figures illustrating the sampling results, the black dashed lines designate the boundary of the exact obstacle.

\begin{example}[Sound-soft obstacles] 
In the first example, we consider the reconstruction of the sound-soft obstacle and the excitation source points with $k=14$. In \Cref{fig: Dirichlet}, the boundary of the exact obstacle is given by:
$$
x(t) = (\cos t+0.65\cos2t-0.65, 1.5\sin t),\quad 0\le t\le 2\pi.
$$

By adding 12 auxiliary source points into the co-inversion setup, we show the reconstructions in \Cref{fig: Dirichlet}(b)-(e). 
In \Cref{fig: Dirichlet}(b), we plot the reconstruction of the obstacle. As can be seen in \Cref{fig: Dirichlet}(b), the value of the indicator function highlights the boundary of the kite-shaped obstacle, while the indicator is relatively small both inside and outside the obstacle. One can further find that both the convex part and the concave part of the obstacle could be well recovered, which validates the effectiveness of our method.

For the reconstruction of the source points, we can see that in either \Cref{fig: Dirichlet}(c) or \Cref{fig: Dirichlet}(d), there are 10 significant local maximizers in $\mathcal{T}_2$ and 5 significant local maximizers in $\mathcal{T}_1,$ respectively. Denoting these local maximizers of \Cref{fig: Dirichlet}(c) by the black `+' markers and the local maximizers of \Cref{fig: Dirichlet}(d) by the blue `$\times$' markers in \Cref{fig: Dirichlet}(e), we compare the locations with the exact source locations (denote by the red small point), which demonstrates source locations can be well identified.
From \Cref{fig: Dirichlet}(c)-(e), we can see that all the functions $I_1$, $I_2$ and $\widehat{I_2}$ are capable of indicating the source points. For clarity, we only plot the imaging results of $I_1$ and $\widehat{I_2}$ in the subsequent examples to show the reconstructions.

\begin{figure} 
	\centering  
	\subfigure[]{\includegraphics[width=0.3\textwidth]{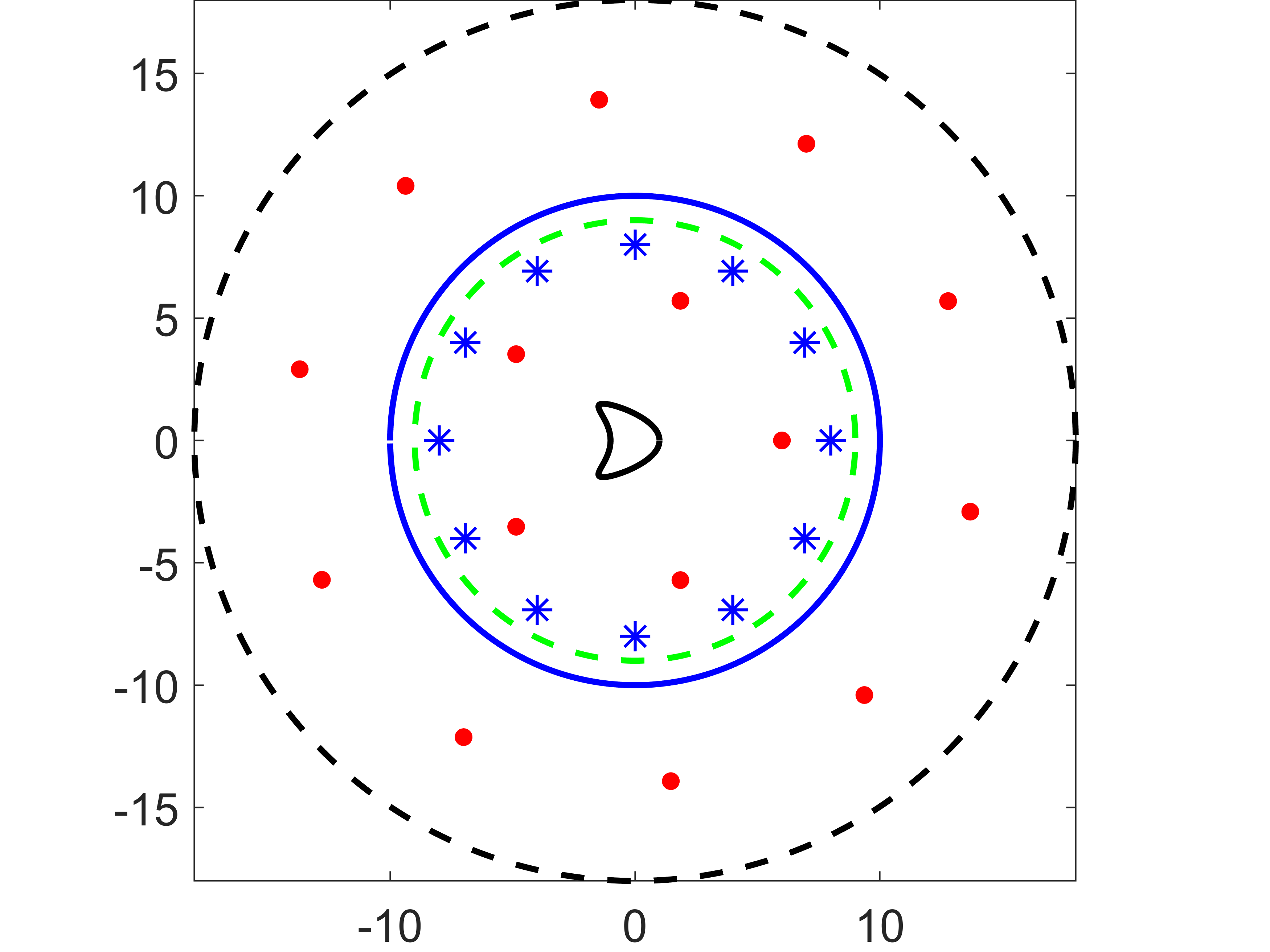}}\quad
	\subfigure[]{\includegraphics[width=0.3\textwidth]{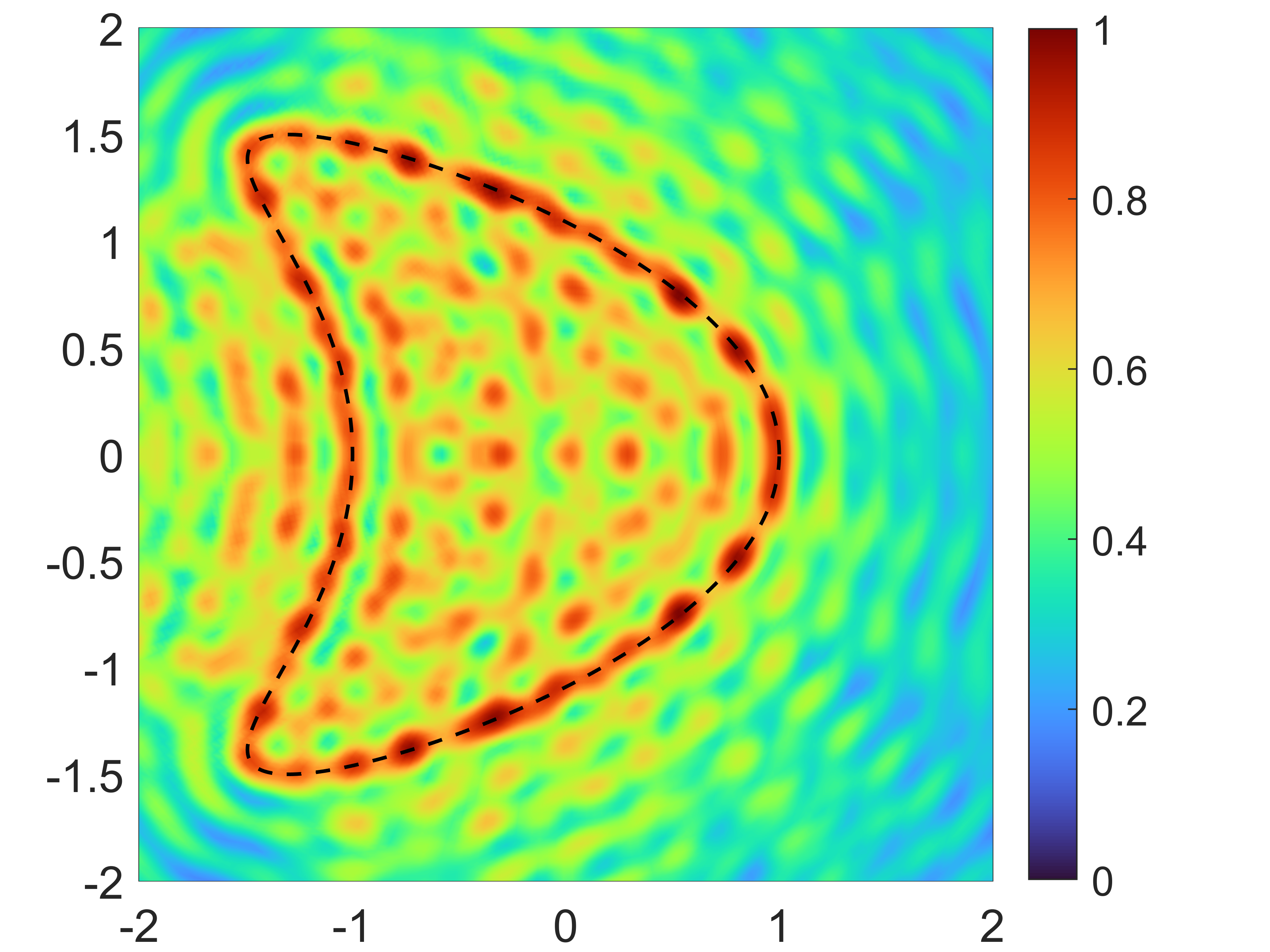}}\\
	\subfigure[]{\includegraphics[width=0.3\textwidth]{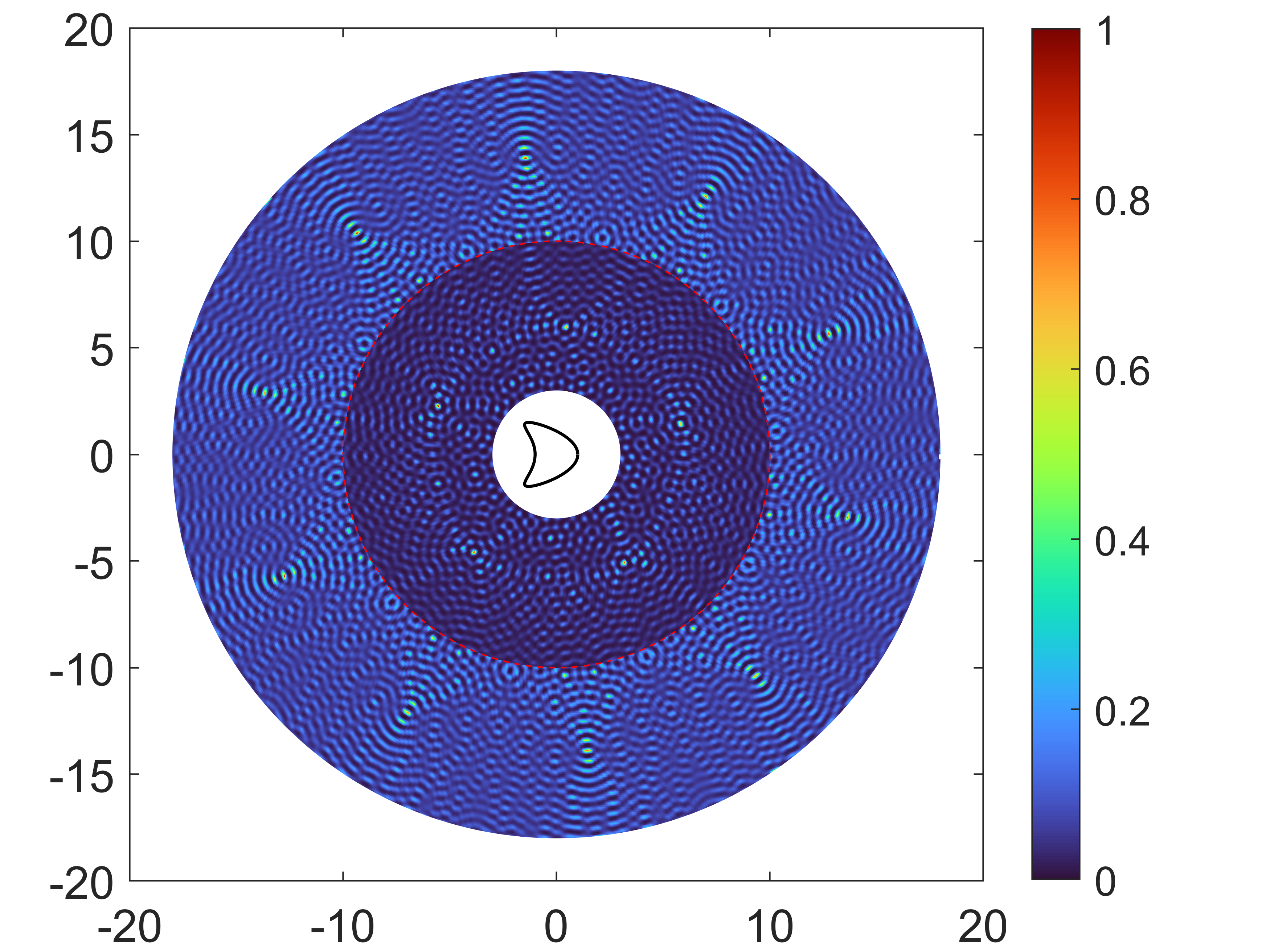}}\quad
	\subfigure[]{\includegraphics[width=0.3\textwidth]{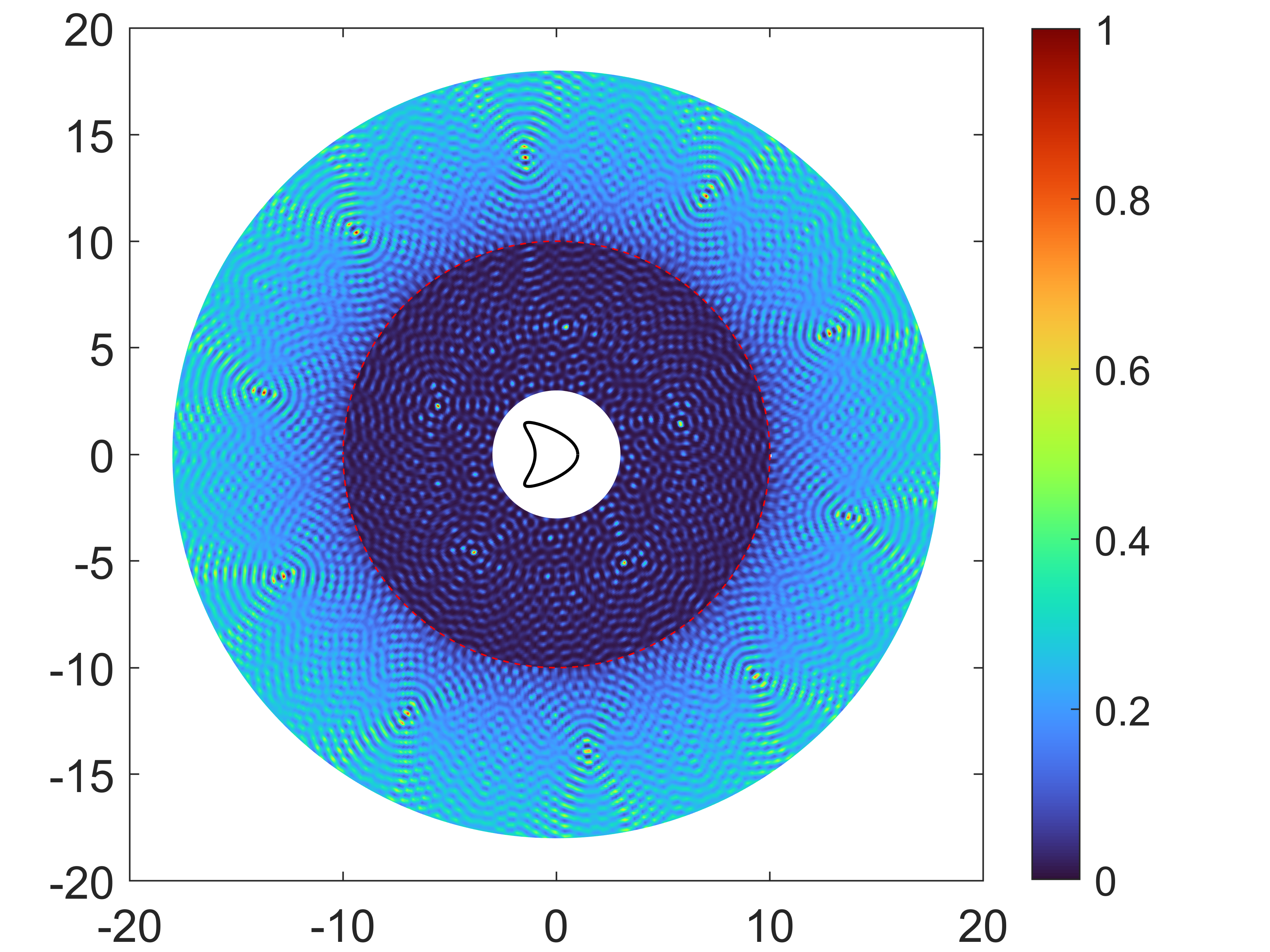}}\quad
	\subfigure[]{\includegraphics[width=0.3\textwidth]{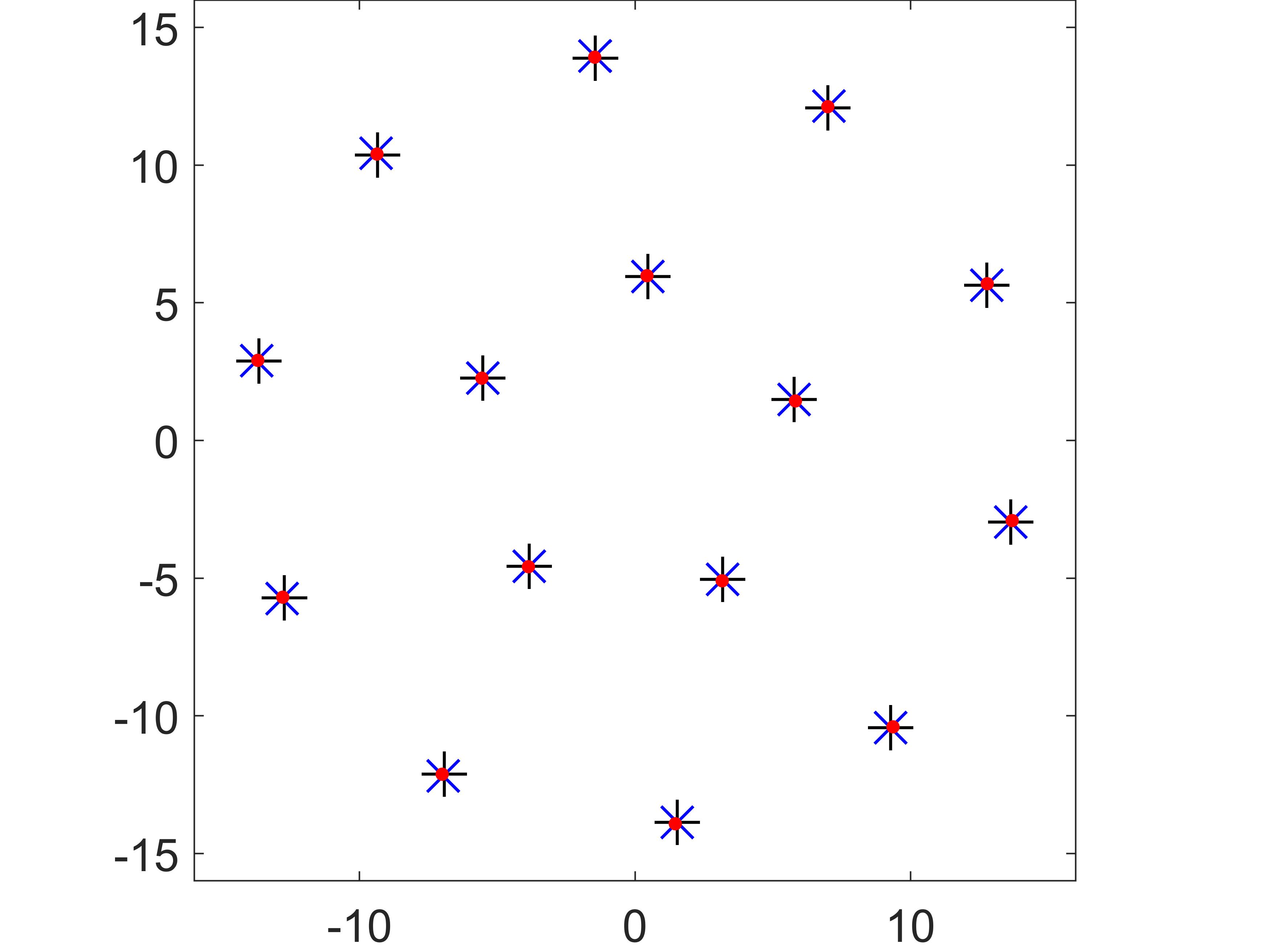}}
	\caption{Reconstruction of a sound-soft kite and the source points. (a) model setup; (b) imaging of the obstacle; (c) imaging of the $I_1(y)$ and $I_2^{-1}(y)$; (d) imaging of the $I_1(y)$ and $\widehat{I}_2(y)$; (e) comparison of the source recovery.}\label{fig: Dirichlet}
\end{figure}

Next, we consider the reconstruction of the sound-soft obstacle with two disjointed components and their excitation sources. The two components of the boundary can be respectively parameterized by
\begin{align*}
  x_1(t) & = (0.5\cos t+0.325\cos2t+0.675,0.75\sin t-1),\quad0\le t\le 2\pi,\\
  x_2(t) & = (\cos t+0.65\cos2t-1.65,1.5\sin t+1),\quad0\le t\le 2\pi.
\end{align*} 
The resulting reconstructions are shown in \Cref{fig: Dirichlet2}. The scatterer consists of two kite-shaped obstacles of different scales and orientations. From \Cref{fig: Dirichlet2}(b)-(d), we find that all the source points and the obstacles are well reconstructed. This example shows that the proposed method has the capability of imaging multiple obstacles and multiple sources.

\begin{figure} 
	\centering  
	\subfigure[]{\includegraphics[width=0.21\textwidth]{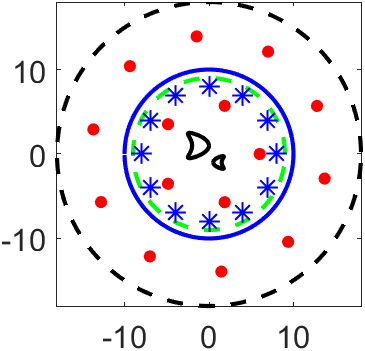}}\quad
	\subfigure[]{\includegraphics[width=0.26\textwidth]{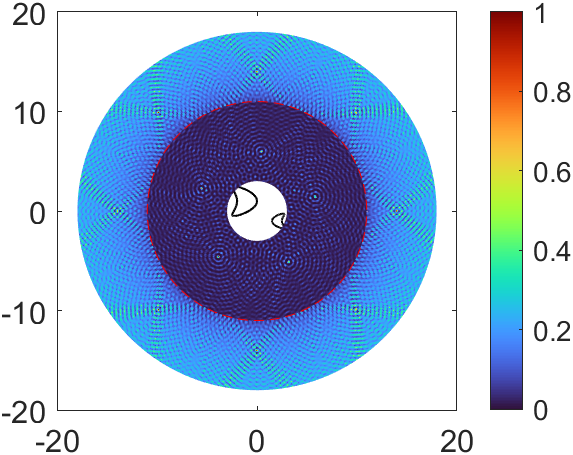}}\quad
	\subfigure[]{\includegraphics[width=0.21\textwidth]{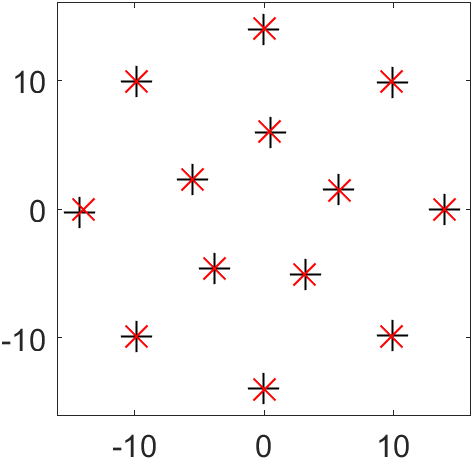}}\quad
	\subfigure[]{\includegraphics[width=0.25\textwidth]{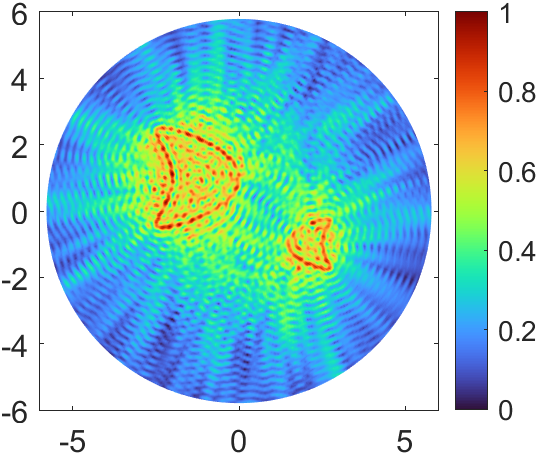}}
	\caption{Reconstruction of two sound-soft kites and the source points. (a) model setup; (b) imaging of the source points; (c) comparison of the source recovery; (d) imaging of the obstacle.}\label{fig: Dirichlet2}
\end{figure}

Furthermore, the two obstacles and the source points are recovered from various limited-aperture observations in \Cref{fig: Dirichlet2i}. As shown in \Cref{fig: Dirichlet2i}(a)(e)(i), only partial data is available in this setup. In this case, the source points and the obstacles that are adequately illuminated can be better reconstructed. In comparison, the non-illuminated portion of the targets (both the source and the obstacle) is less accurately retrieved due to the lack of information.
\end{example}

\begin{figure} 
	\centering  
	\subfigure[]{\includegraphics[width=0.21\textwidth]{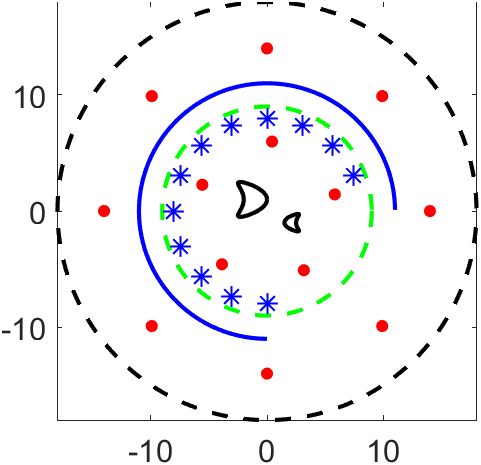}}\quad
	\subfigure[]{\includegraphics[width=0.26\textwidth]{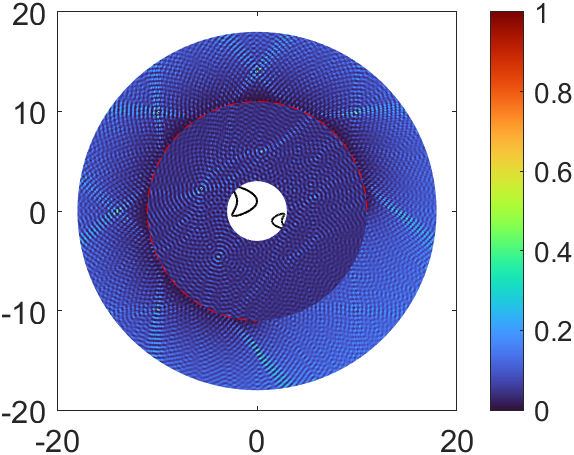}}\quad
	\subfigure[]{\includegraphics[width=0.21\textwidth]{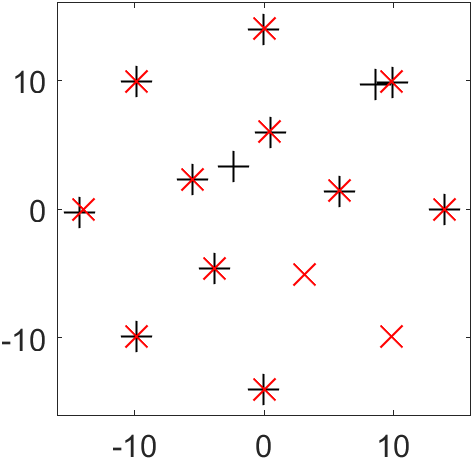}}\quad
	\subfigure[]{\includegraphics[width=0.25\textwidth]{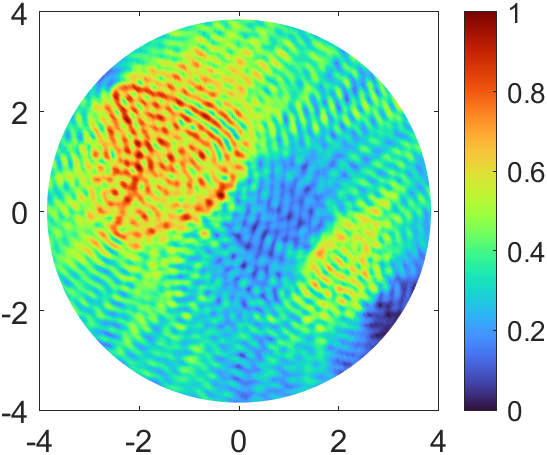}}\\

	\subfigure[]{\includegraphics[width=0.21\textwidth]{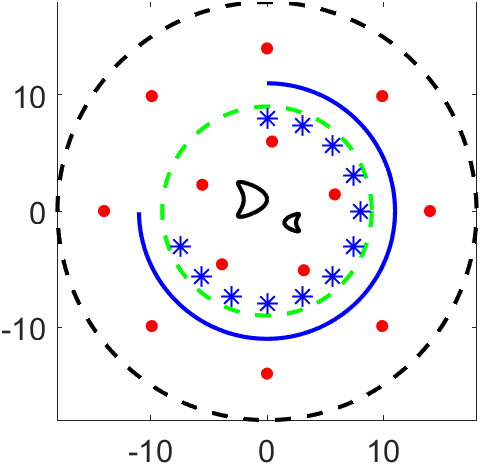}}\quad
	\subfigure[]{\includegraphics[width=0.26\textwidth]{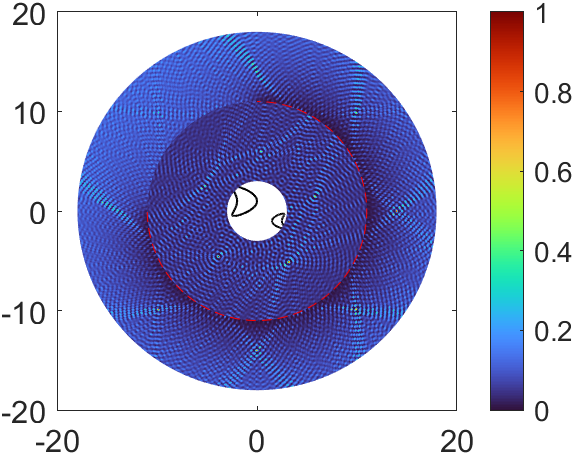}}\quad
	\subfigure[]{\includegraphics[width=0.21\textwidth]{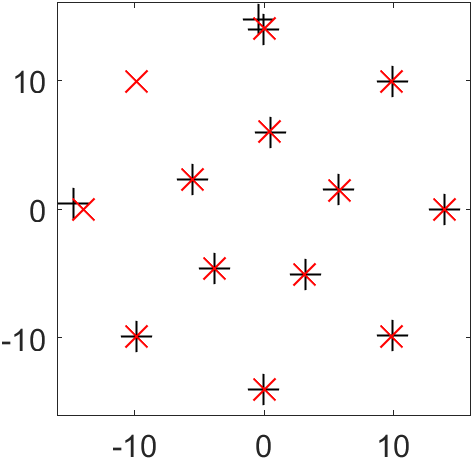}}\quad
	\subfigure[]{\includegraphics[width=0.25\textwidth]{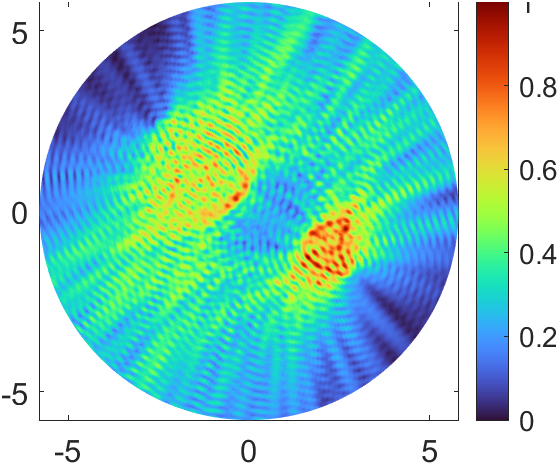}}\\
		
	\subfigure[]{\includegraphics[width=0.21\textwidth]{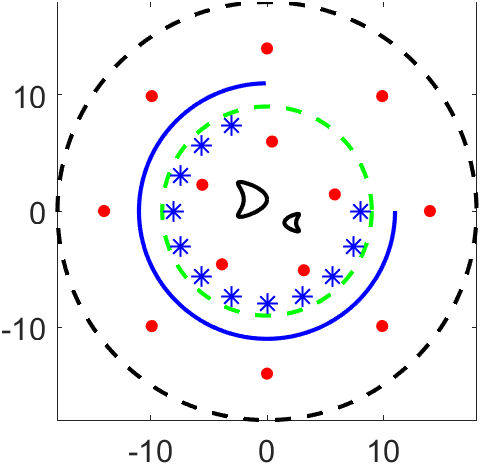}}\quad
	\subfigure[]{\includegraphics[width=0.26\textwidth]{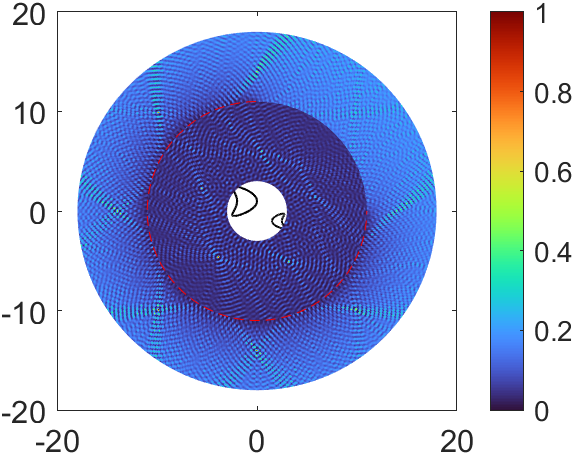}}\quad
	\subfigure[]{\includegraphics[width=0.21\textwidth]{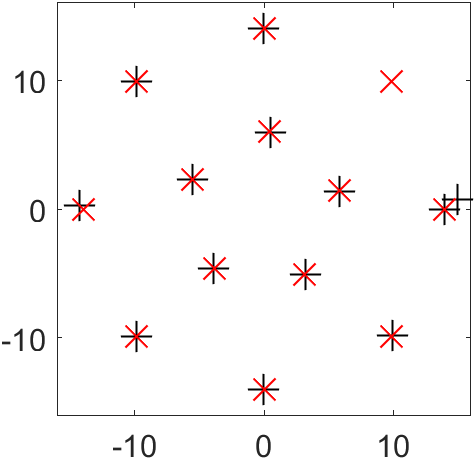}}\quad
	\subfigure[]{\includegraphics[width=0.25\textwidth]{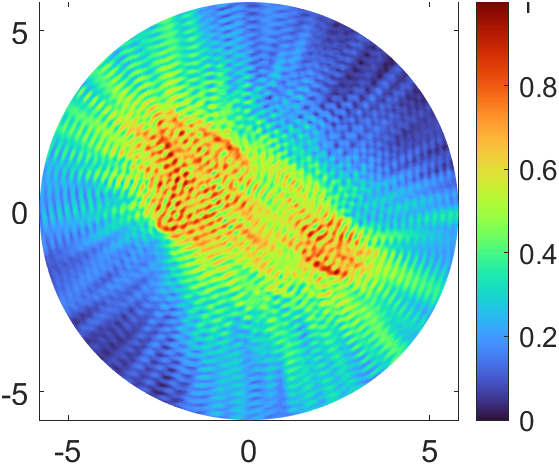}}
	\caption{Reconstruction of two sound-soft kites and the source points with limited aperture data. Rows 1-3: the three cases of limited data;  Column 1: model setup; Column 2:  imaging of the source points; Column 3: comparison of the source recovery; Column 4: imaging of the obstacles.}\label{fig: Dirichlet2i}
\end{figure}

%------------------------------------------------------------------

\begin{example}[Sound-hard obstacles]
The aim of the second example is to test the performance of the sampling methods by reconstructing the sound-hard obstacle and the source points. The boundary of the obstacle is parameterized by
\[
x(t)=0.5(\cos^3t+\cos t,\sin^3t+\sin t),\quad 0\leq t\le 2\pi.
\]

The reconstructions with $k=6$ and $k=12$ are displayed in \Cref{fig:Neumann}, which shows that the shape of the obstacle and the source locations can be identified by the proposed method. In \Cref{fig:Neumann}(b) and \Cref{fig:Neumann}(e), we can see that the indicator function $I_1(y)$ has 5 significant local maximizers in $\mathcal{T}_1$ while $\widetilde{I}_2(y)$ has 6 significant local maximizers in $\mathcal{T}_2$. All the significant local maximizers in $\mathcal{T}_1$ and $\mathcal{T}_2$ compiles well with the exact locations. In \Cref{fig:Neumann}(c) and \Cref{fig:Neumann}(f), we display the imaging of the indicator $I_D(y).$ We can see that $I_D(y)$ attains its local minimizers exactly on the boundary of the obstacle, which illustrates that our method performs well in identifying the boundary of the sound-hard obstacle. 

Next, we test the influence of the quantity of the auxiliary sources on the reconstruction of the obstacle. We can see from \cref{fig:Neumann2} that as the number of auxiliary sources increases, the resolution of the sound-hard obstacle could be enhanced. When the auxiliary sources is relatively small, we may find several apparent local minimizers inside the obstacle, and it seems relatively difficult to capture the boundary exactly from the imaging figure. When we increase the auxiliary sources, we can easily locate the boundary where the indicator attains its local minimizer.
\end{example}

\begin{figure}\centering
	\subfigure[]{\includegraphics[width=0.26\textwidth]{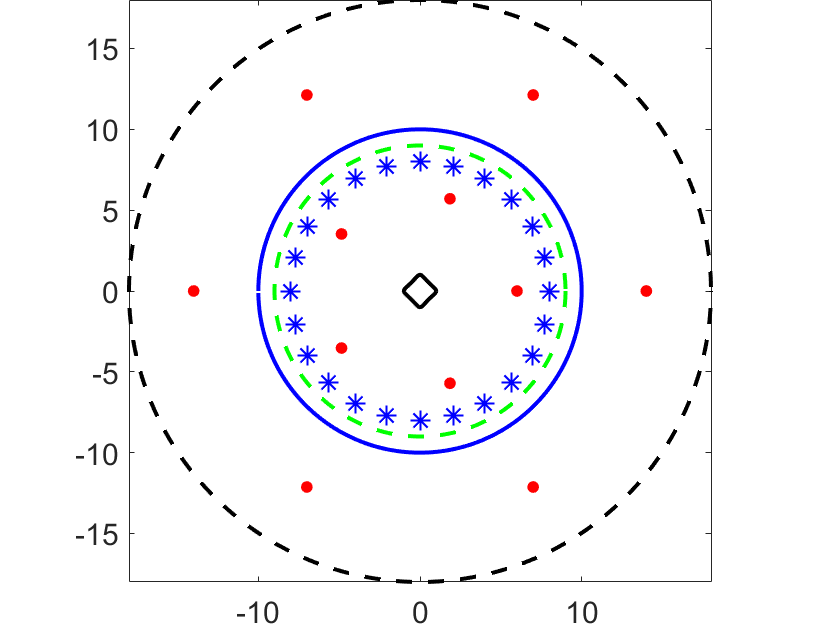}}\quad
	\subfigure[]{\includegraphics[width=0.32\textwidth]{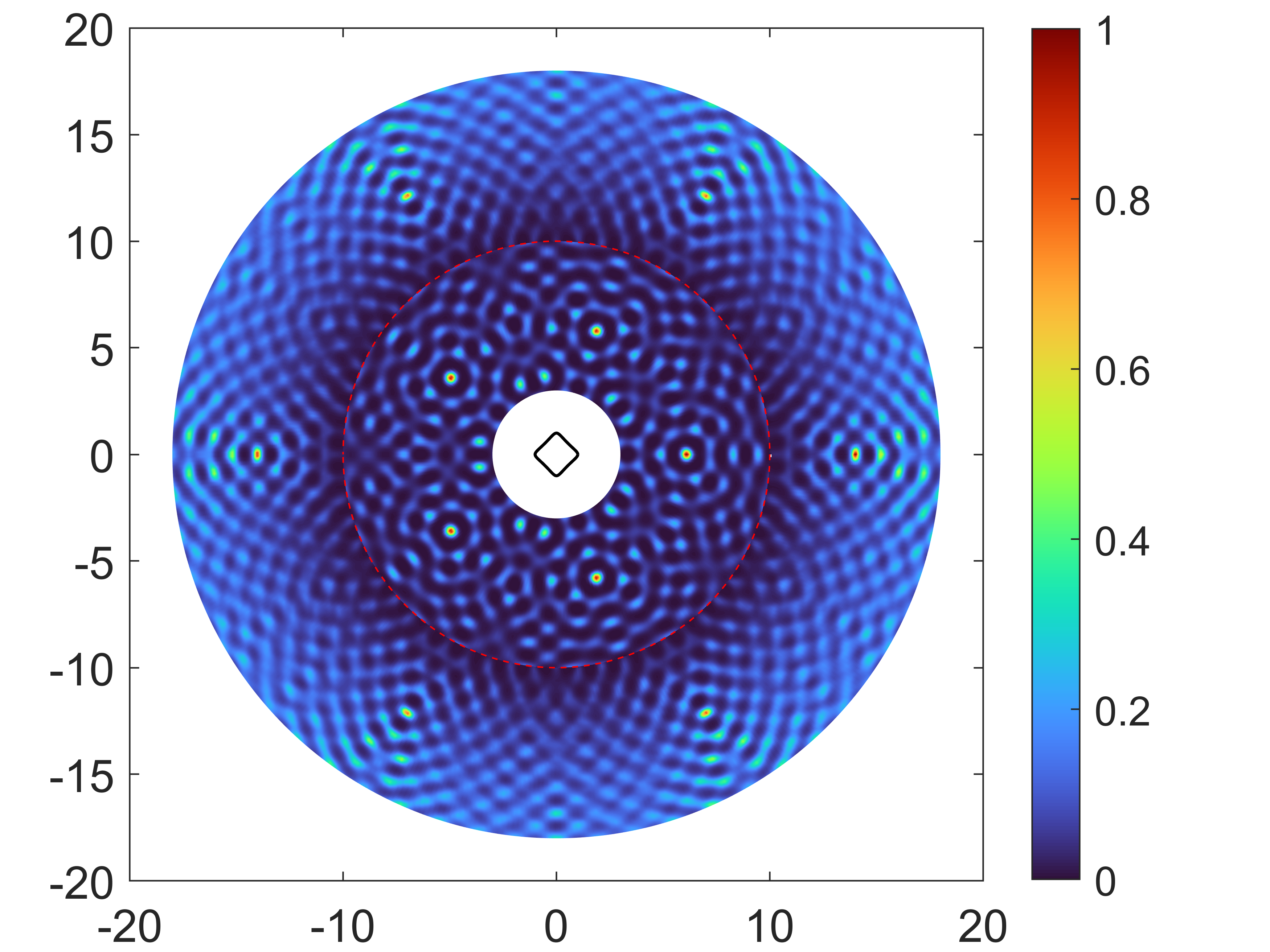}}\quad
	\subfigure[]{\includegraphics[width=0.3\textwidth]{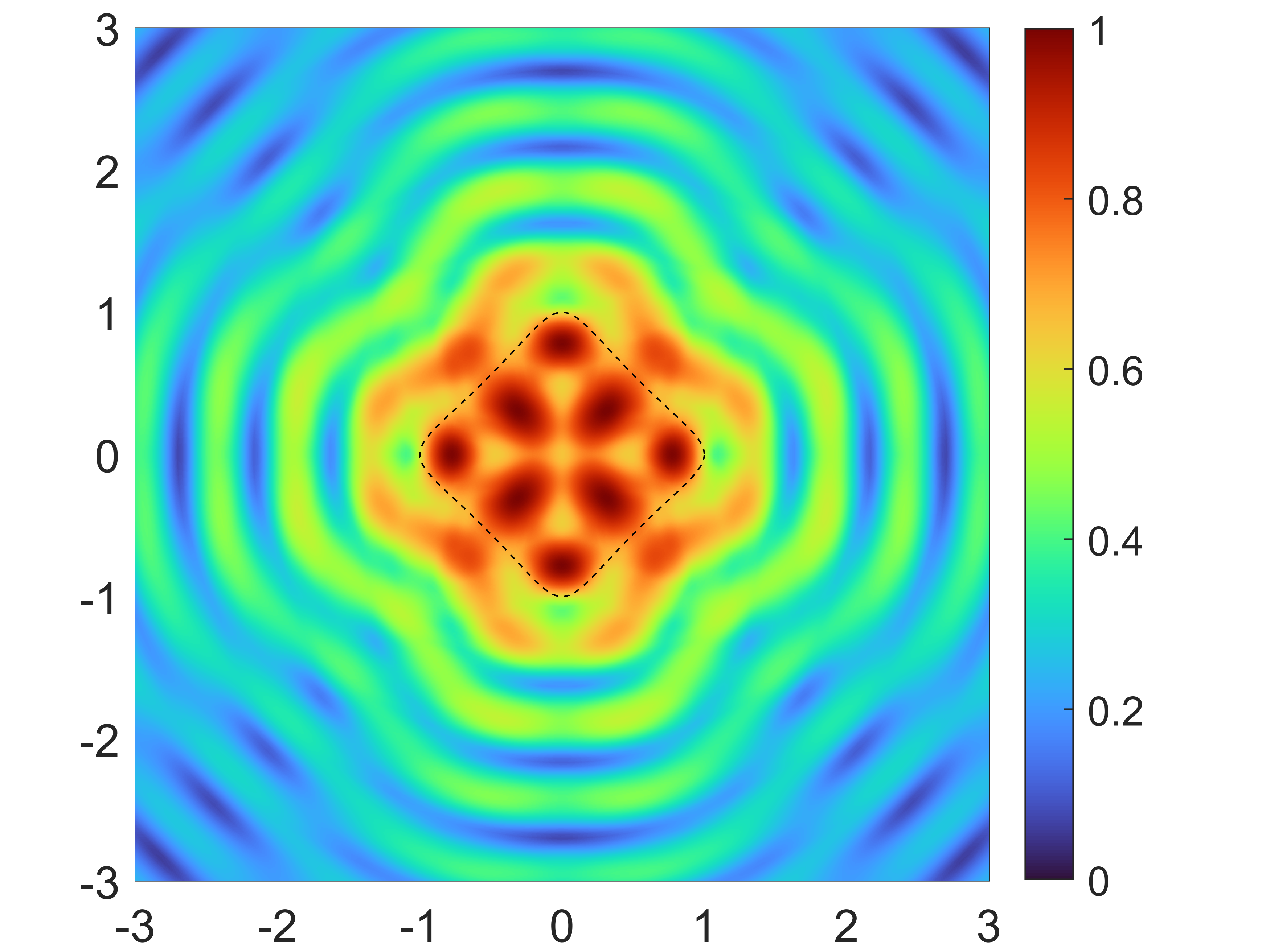}}\\
	\subfigure[]{\includegraphics[width=0.26\textwidth]{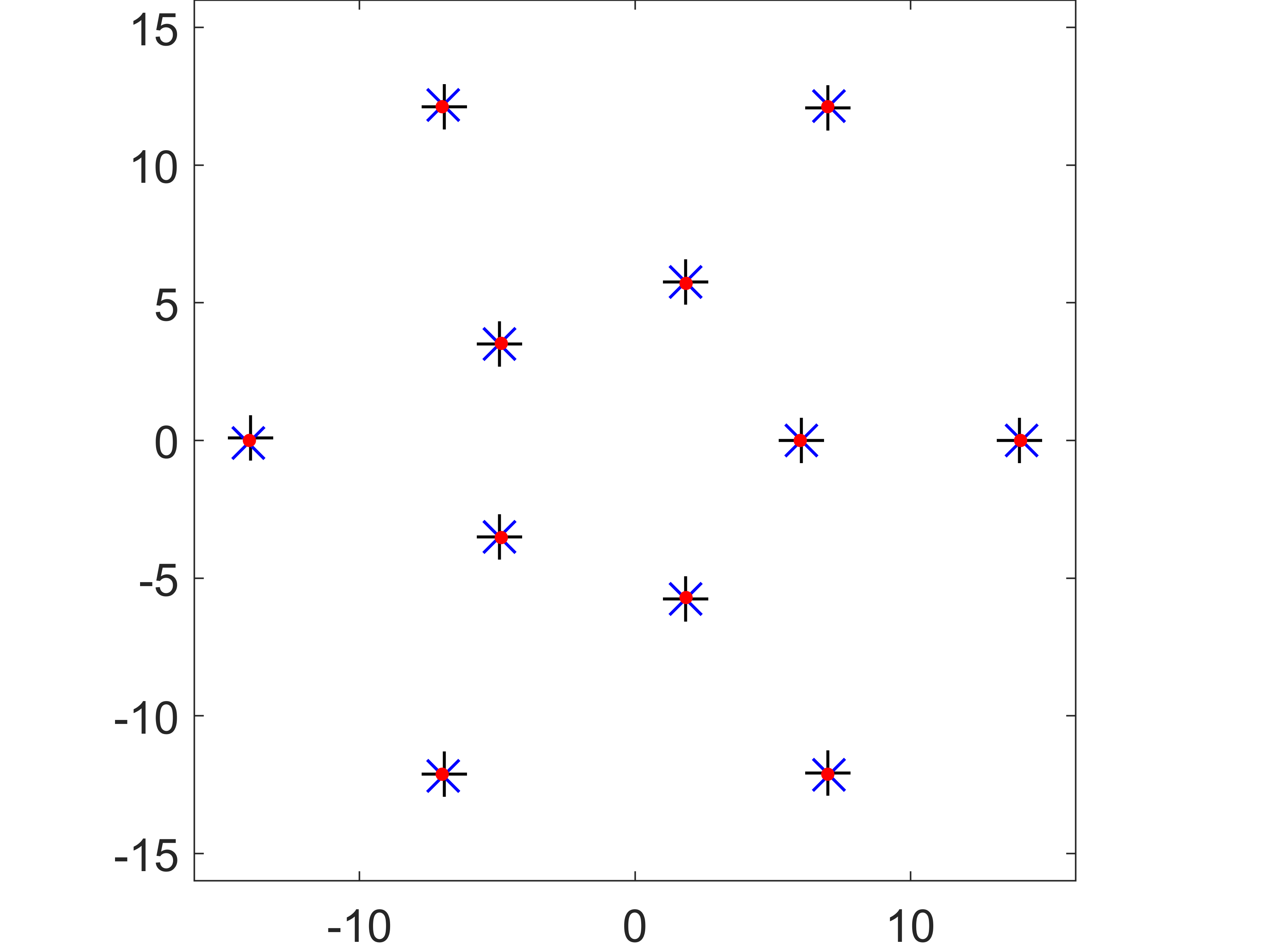}}\quad
	\subfigure[]{\includegraphics[width=0.32\textwidth]{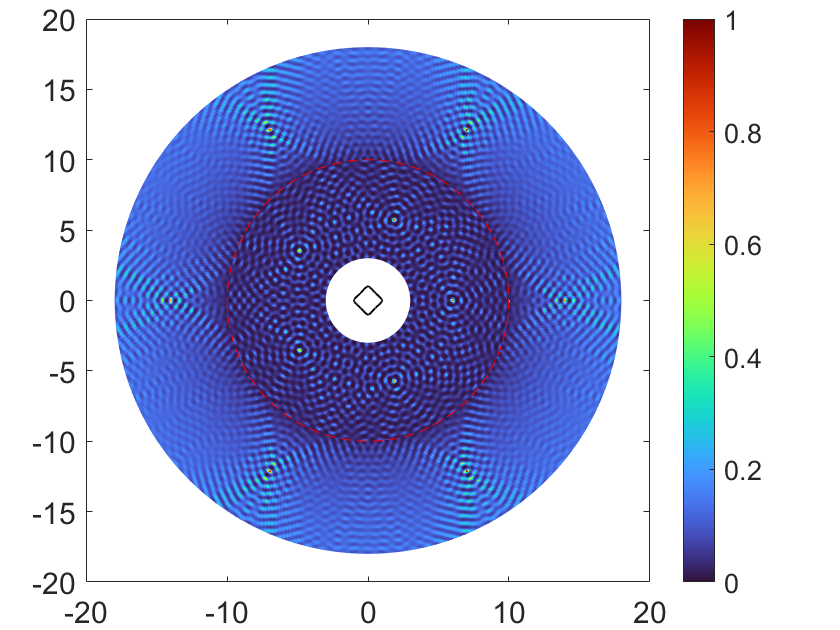}}\quad
	\subfigure[]{\includegraphics[width=0.3\textwidth]{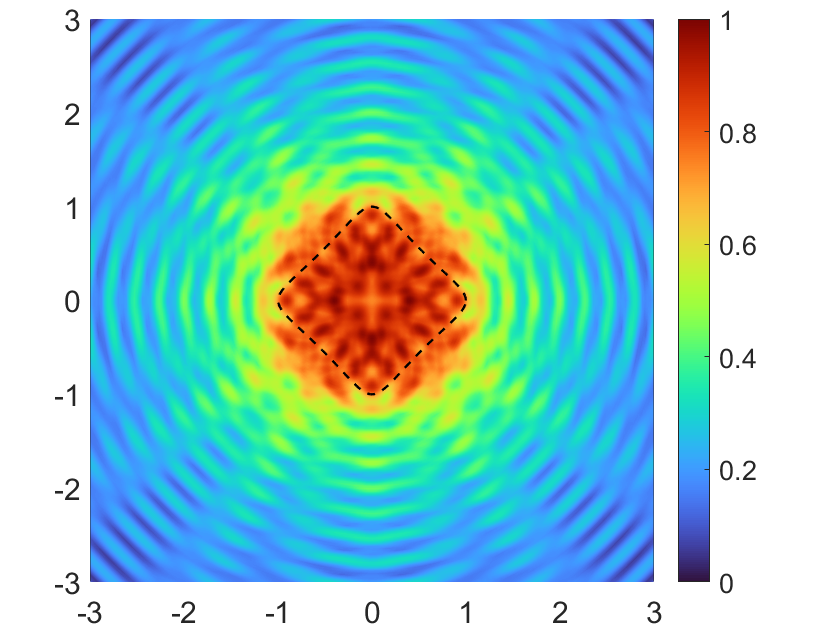}}
	\caption{Reconstruction of a sound-hard obstacle and the excitation source points subject to different wavenumbers. (a) model setup; (b)-(d): reconstructions with $k=6;$ (e)-(f):  reconstructions with $k=12.$}\label{fig:Neumann}
\end{figure}

\begin{figure}
    \centering
    \subfigure[]{\includegraphics[width=0.29\textwidth]{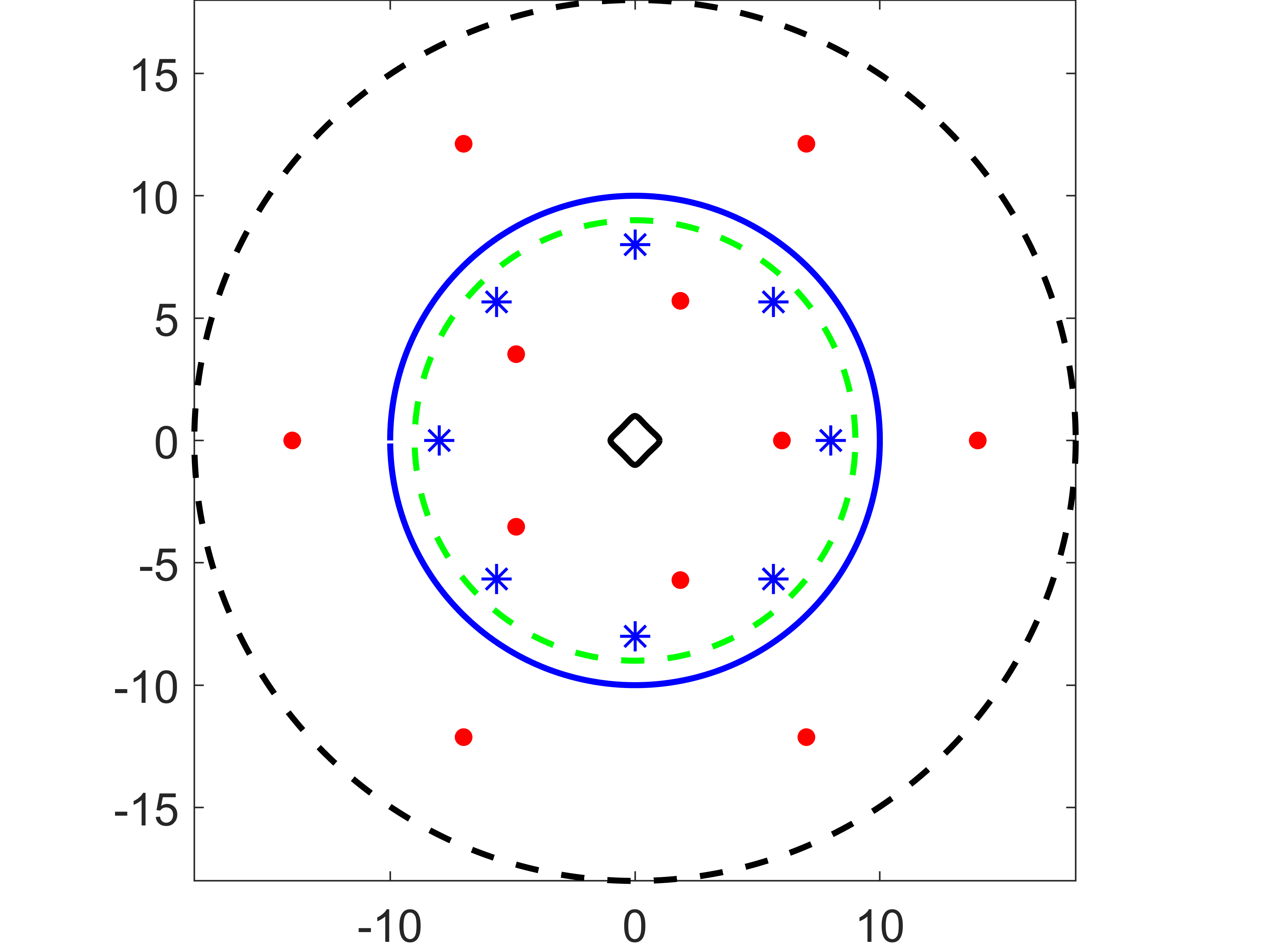}}\qquad
    \subfigure[]{\includegraphics[width=0.29\textwidth]{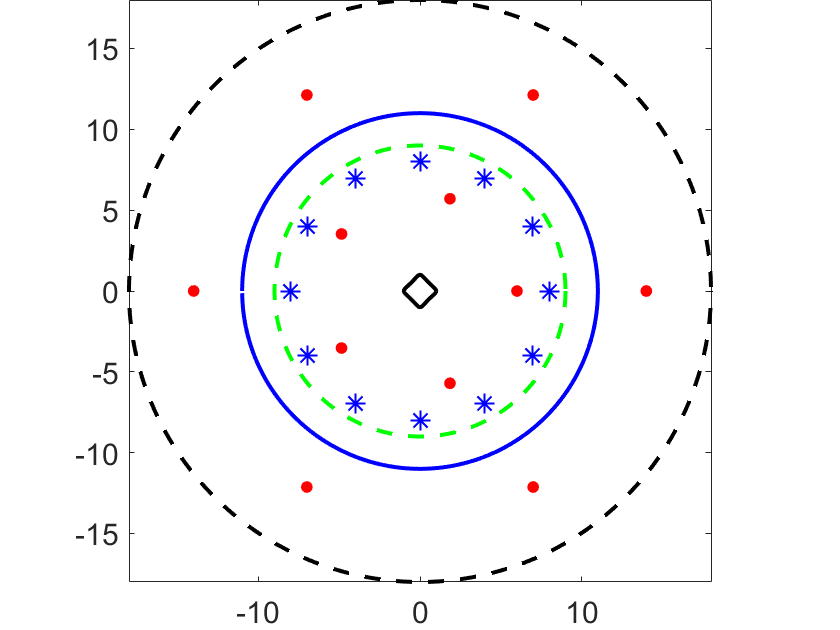}}\qquad
    \subfigure[]{\includegraphics[width=0.29\textwidth]{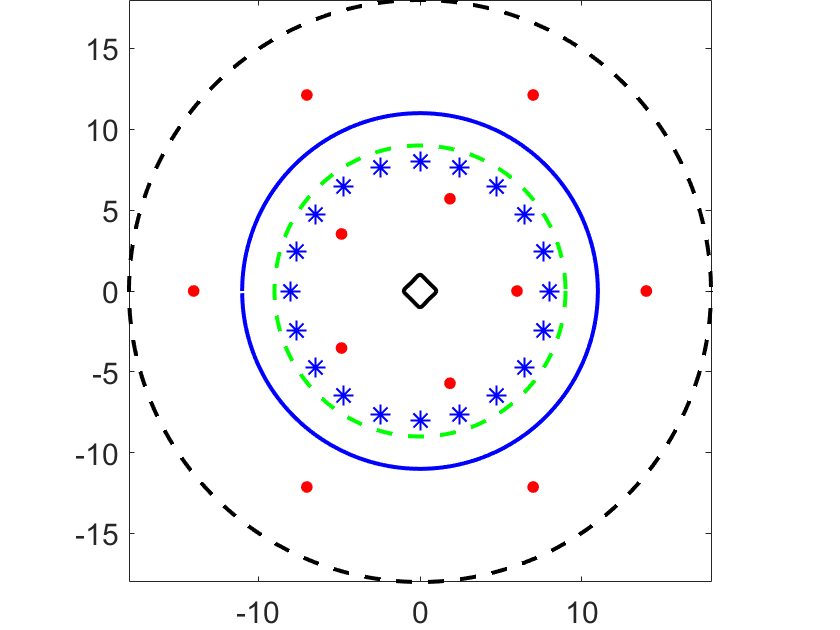}}\\ \quad
    \subfigure[]{\includegraphics[width=0.31\textwidth]{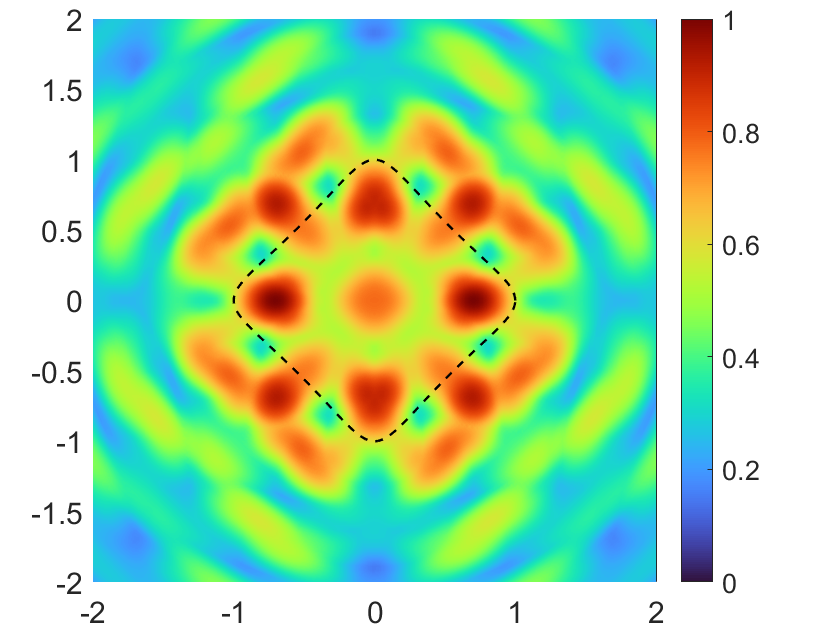}}\quad
    \subfigure[]{\includegraphics[width=0.31\textwidth]{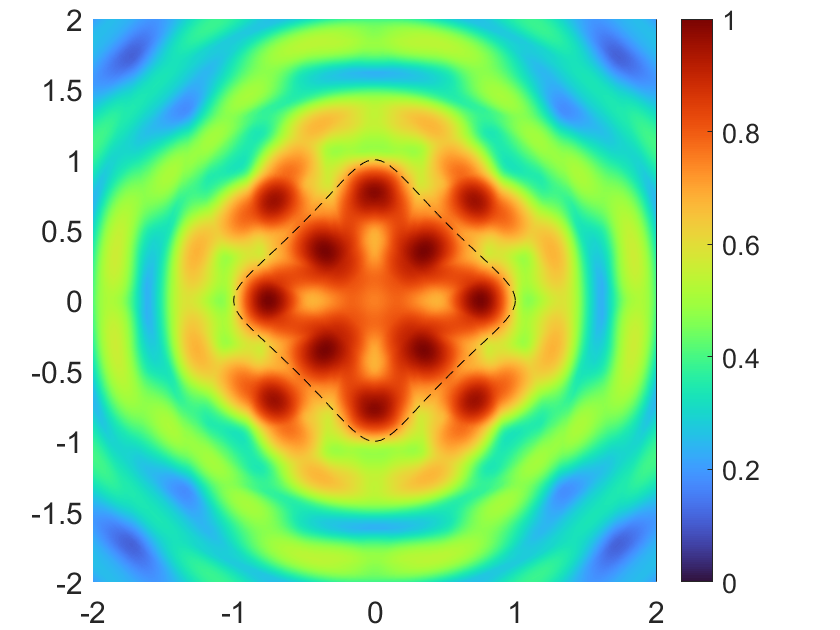}}\quad
    \subfigure[]{\includegraphics[width=0.31\textwidth]{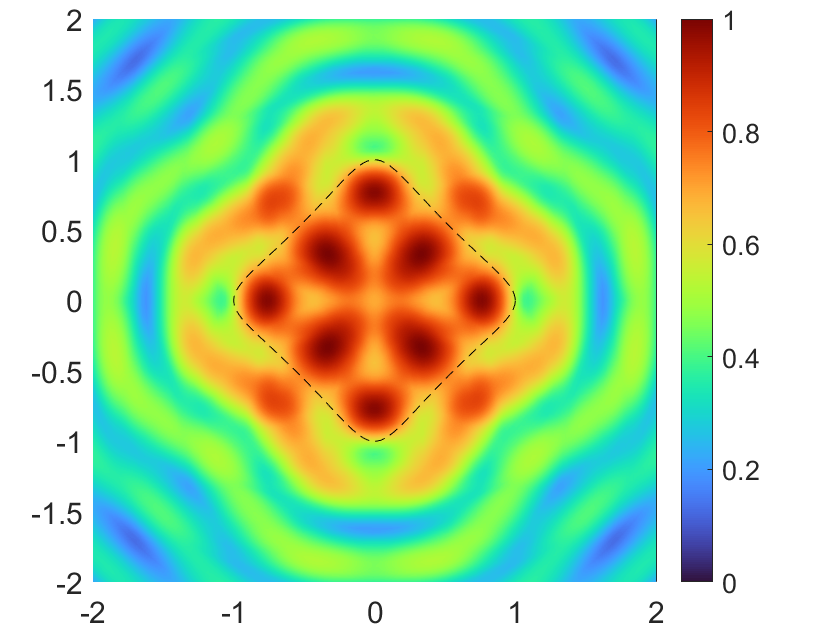}}
\caption{Reconstruction of a sound-hard obstacle by the different number of auxiliary source points ($k=6$). Row 1: model setups; Row 2: reconstructions.}\label{fig:Neumann2}
\end{figure}

%-----------------------------------------------------------------

\begin{example}[Multiscale obstacles]
The scatterer in this example consists of a sound-soft kite parameterized by
$$
x(t) = (\cos t-0.4\cos4t,\sin t),\quad 0\le t\le 2\pi,
$$
and a sound-hard disk centered at $(1,-1)$ with radius $r=0.1$. Note that the two scattering components are of significant geometrical scale. Moreover, the smaller component is very close to the larger one, meanwhile, two of the source points are close to each other. Apparently, these features of configuration lead to difficulties in resolving them. We consider the reconstruction of these multiscale obstacles and the source points simultaneously.  Now $24$ auxiliary source points are supplemented to the co-inversion problem. To ease the visualization of the true obstacles, the local part $[-2,2]\times[-2,2]$ in \cref{fig:multiscale}(a) is zoomed in and plotted in \cref{fig:multiscale}(b). Numerical results in \cref{fig:multiscale} demonstrate that a low frequency may not be sufficient to accurately identify a small scatterer or recognize the adjacent point sources. Hence, a moderately large wavenumber is necessary to image the details of the targets. We consider the results of full aperture measurements in \cref{fig:multiscale} and the results of limited aperture measurements in \cref{fig: limited}, respectively.
\end{example}

\begin{figure} 
	\centering  
	\subfigure[]{\includegraphics[width=0.3\textwidth]{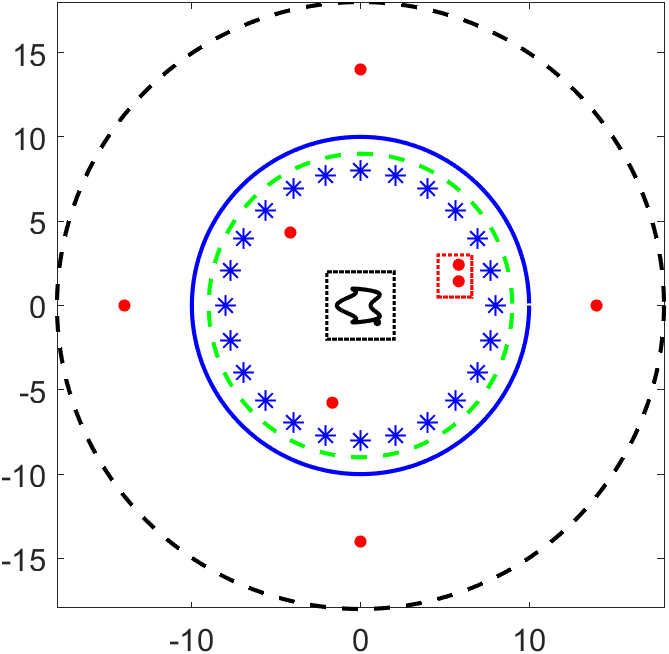}}\quad 
	\subfigure[]{\includegraphics[width=0.3\textwidth]{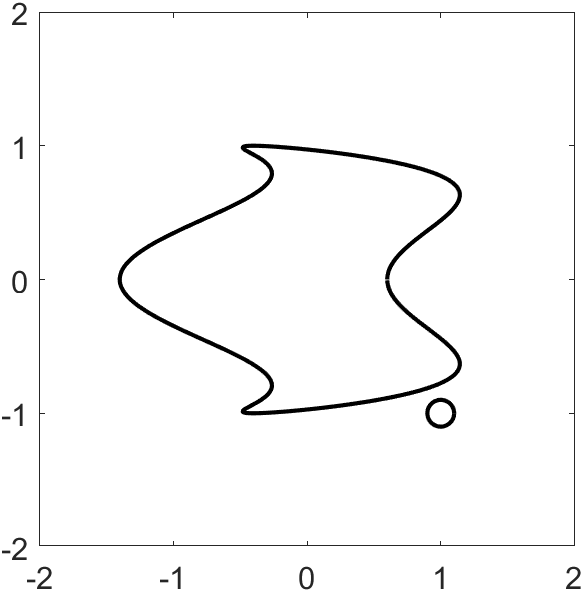}}\\
	\subfigure[]{\includegraphics[width=0.32\textwidth]{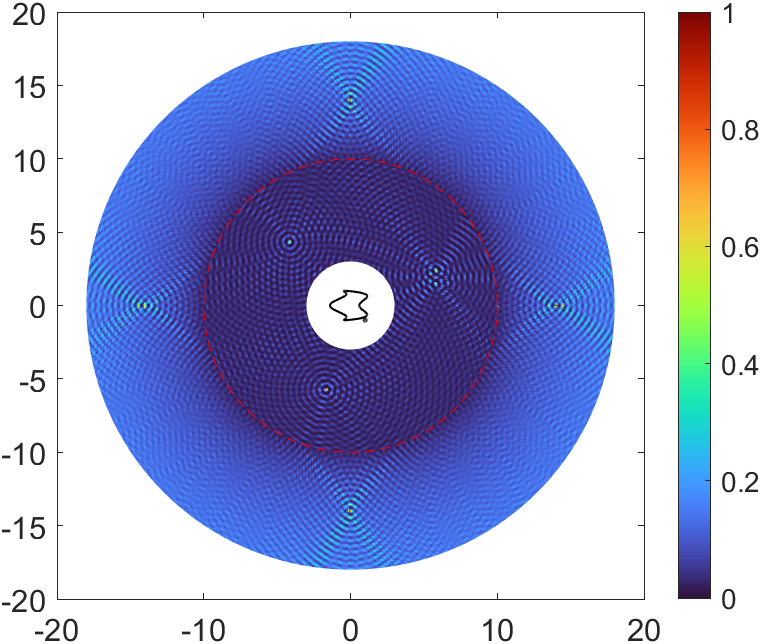}}\quad
	\subfigure[]{\includegraphics[width=0.27\textwidth]{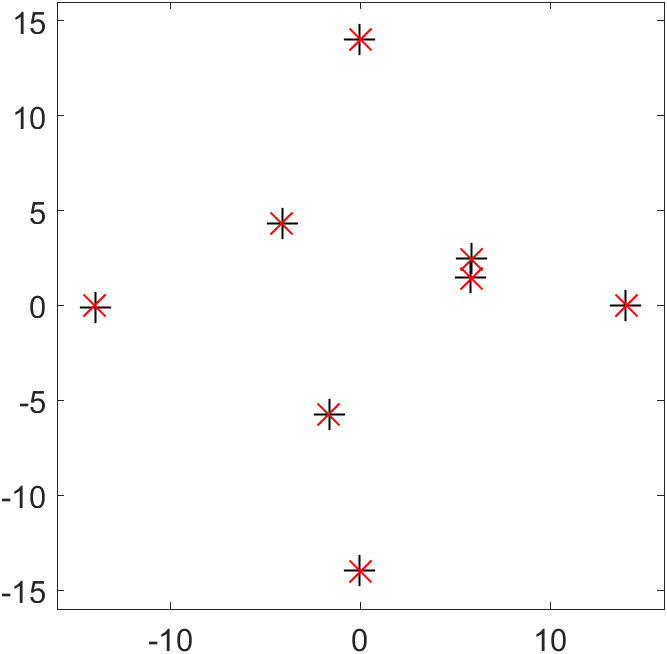}}\quad
	\subfigure[]{\includegraphics[width=0.32\textwidth]{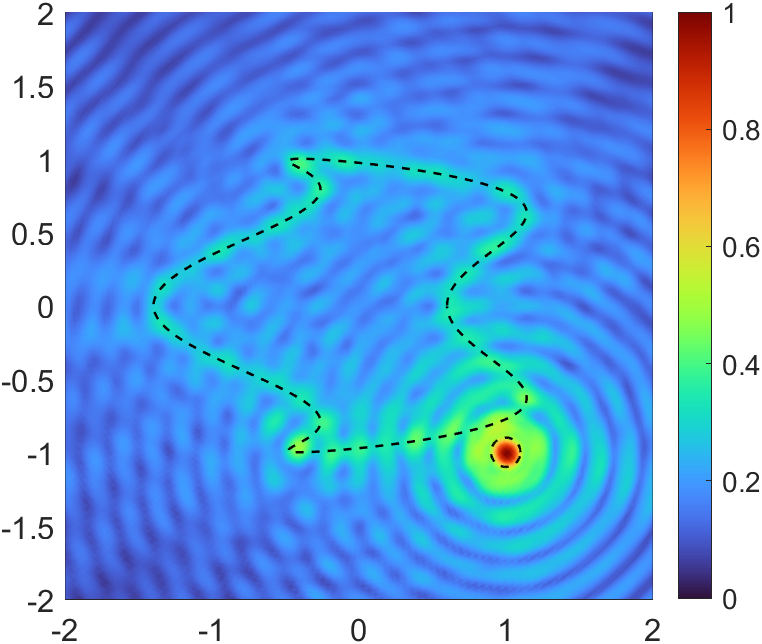}}
	\subfigure[]{\includegraphics[width=0.32\textwidth]{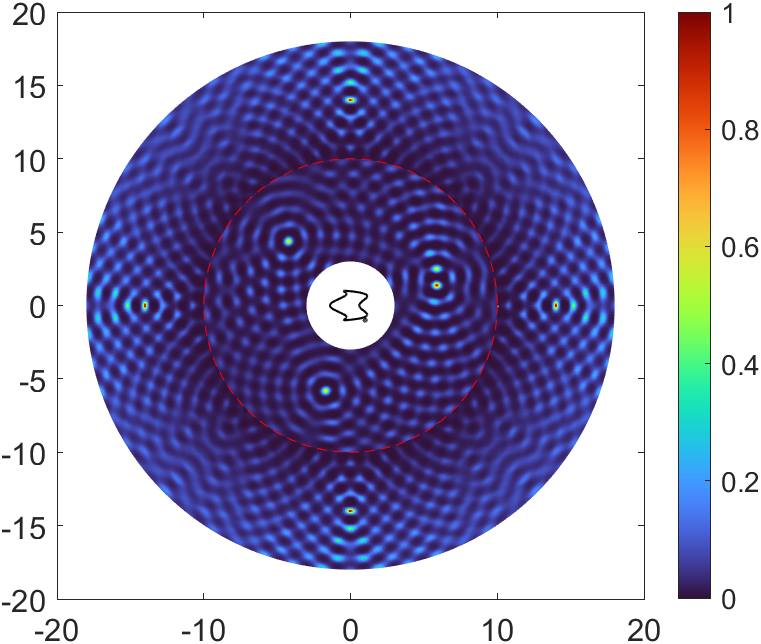}}\quad
	\subfigure[]{\includegraphics[width=0.27\textwidth]{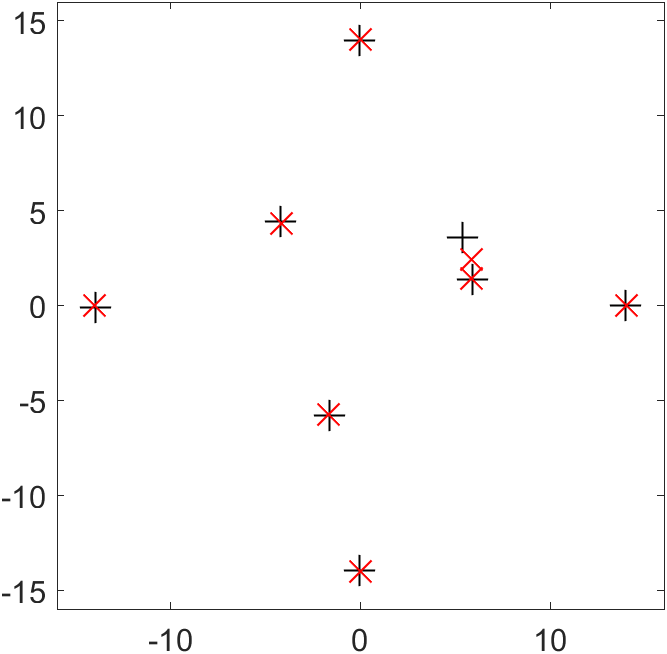}}\quad
	\subfigure[]{\includegraphics[width=0.32\textwidth]{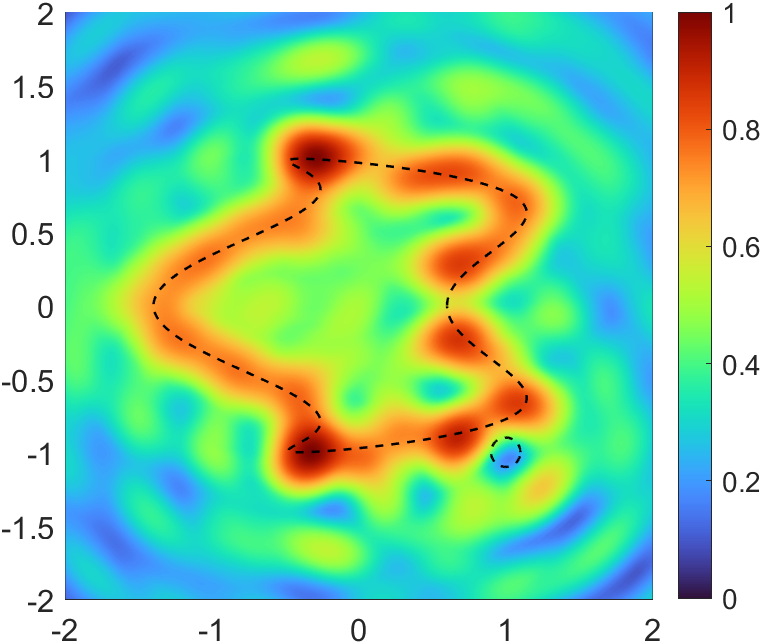}}
	\caption{Reconstruction of the multi-scale obstacles and 8 sources from full-aperture measurements. (a)-(b) model setups; (c)-(e) results with $k=16$; (f)-(h) results with $k=6$.}\label{fig:multiscale}
\end{figure}

\textbf{Full aperture case.} The reconstructions with $k=16$ are shown in \Cref{fig:multiscale}(c)-(e). In \Cref{fig:multiscale}(c), 8 significant local maximizers can be observed and they can be viewed as the reconstructed source points. Further, we compare the recovered source locations with the exact ones in \Cref{fig:multiscale}(d) and it shows that the source points are well-reconstructed. It deserves noting that, though the source points in the red dashed rectangle (see \Cref{fig:multiscale}(a)) are adjacent to each other, they can still be recognized by our method. We also observe from \Cref{fig:multiscale}(e) that the boundary of each obstacle is clearly captured. 

Next, let $k=6$ and display the reconstructions in \Cref{fig:multiscale}(f)-(h). From \Cref{fig:multiscale}(f), one would clearly see 4 significant local maximizers in $\mathcal{T}_2$ but it seems to be a bit difficult in determining whether 3 or 4 sources are involved in $\mathcal{T}_1$. By collecting the largest 8 local maximizers as the reconstructed source points, we can observe that the sources are very well recovered except for the last one. In other words, the last source point fails to be accurately identified, see \Cref{fig:multiscale}(g). It can be also seen from \Cref{fig:multiscale}(h) that the small disk centered at $(1,-1)$ is not reasonably recovered while the kite-shaped obstacle is well-reconstructed. Hence, a suitably chosen wavenumber is influential in the reconstruction.

\textbf{Limited aperture observations.} We finally consider the co-inversion from the limited aperture observations where only partial data is available. By taking $k=16$ and different incidence and observation apertures, we display the reconstructions in \Cref{fig: limited}. As shown in \Cref{fig: limited}(a), the aperture in the first row is chosen to be $[0, \frac{3\pi}{2}].$  We can see from \Cref{fig: limited}(b)-(c), with relatively less accuracy, the source points can be reconstructed roughly. In \Cref{fig: limited}, we find that the small sound-hard disk is well reconstructed, though the small disk locates at the vacant region where there are no source points, it is well-reconstructed. Meanwhile, the kite can be recognized roughly.

Further, we shrink the observation aperture and consider the case where the sources only illuminate the upper half of the co-inversion model. As expected, the source points located in the region surrounded by the measurement aperture can be better identified whereas those in the shadow region can not be accurately imaged. Surprisingly, the upper sound-soft kite can not be recovered but the lower sound-hard disk turns out to be captured though neither the artificial source points nor the receivers are located in the lower part.

\begin{figure} 
	\centering  
	\subfigure[]{\includegraphics[width=0.21\textwidth]{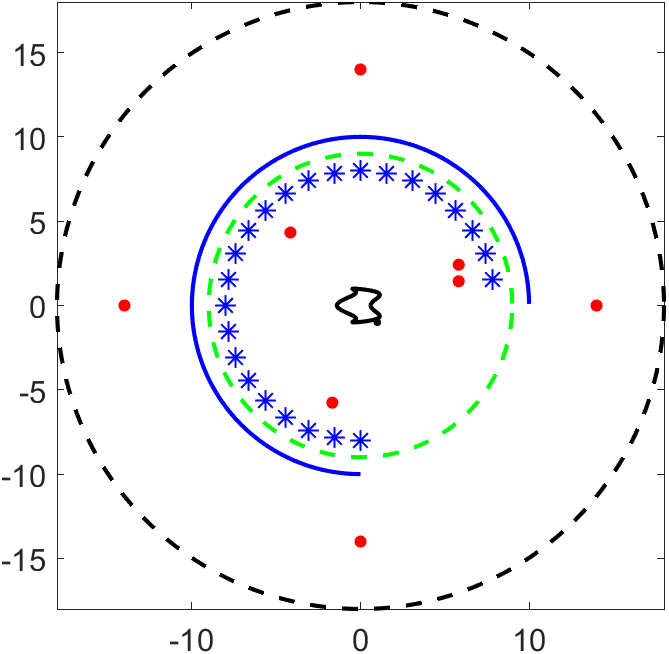}}\quad
	\subfigure[]{\includegraphics[width=0.25\textwidth]{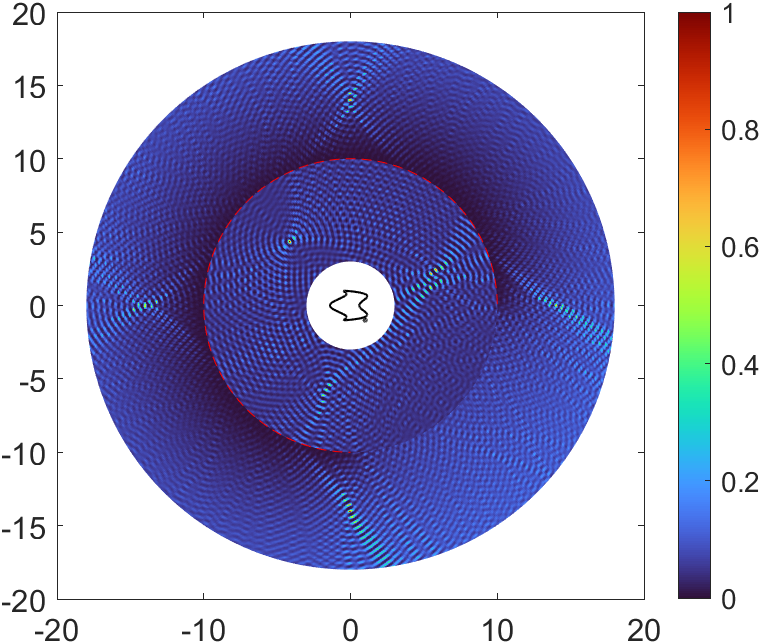}}\quad
	\subfigure[]{\includegraphics[width=0.21\textwidth]{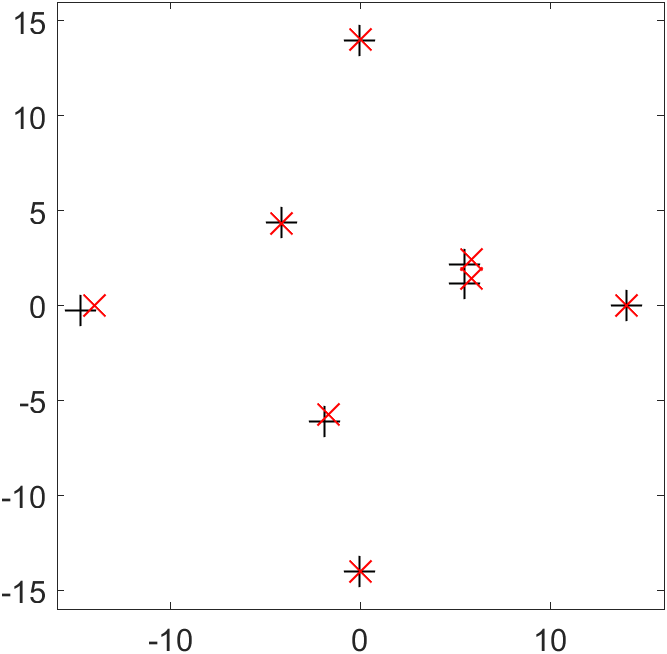}}\quad
	\subfigure[]{\includegraphics[width=0.25\textwidth]{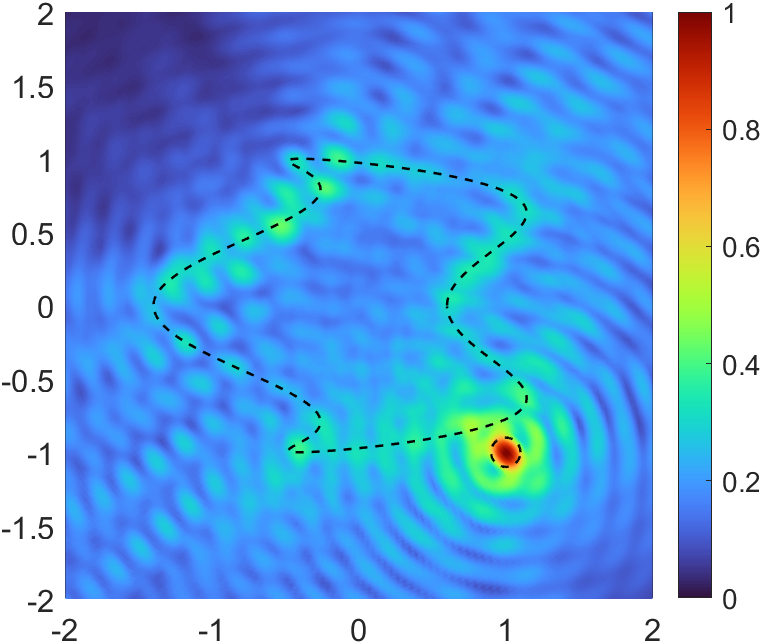}}\\
	\subfigure[]{\includegraphics[width=0.21\textwidth]{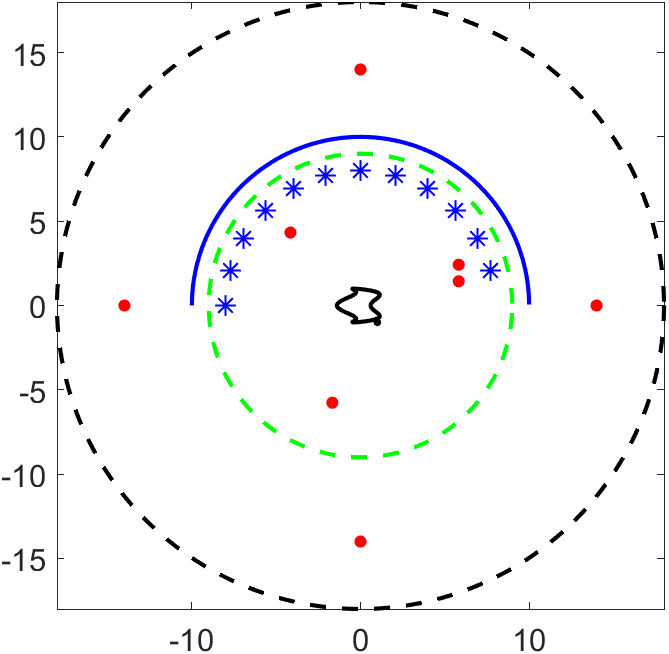}}\quad
	\subfigure[]{\includegraphics[width=0.25\textwidth]{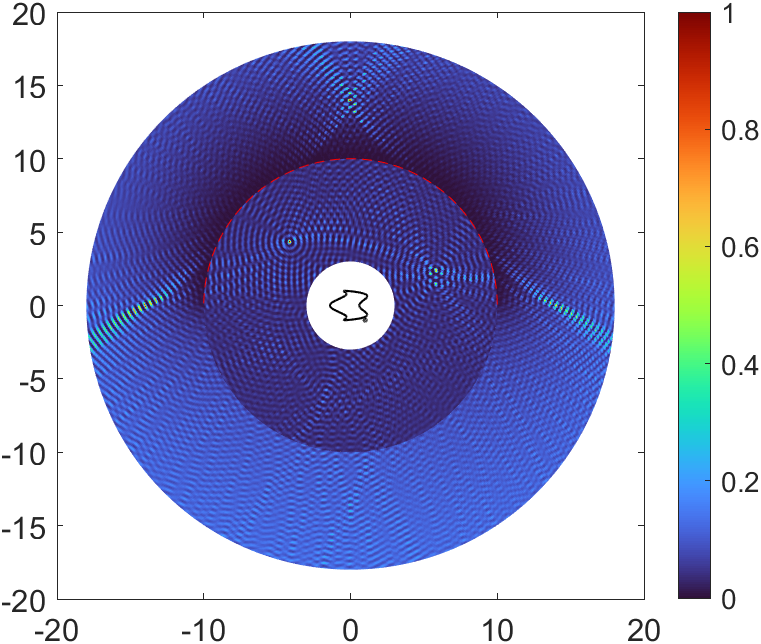}}\quad
	\subfigure[]{\includegraphics[width=0.21\textwidth]{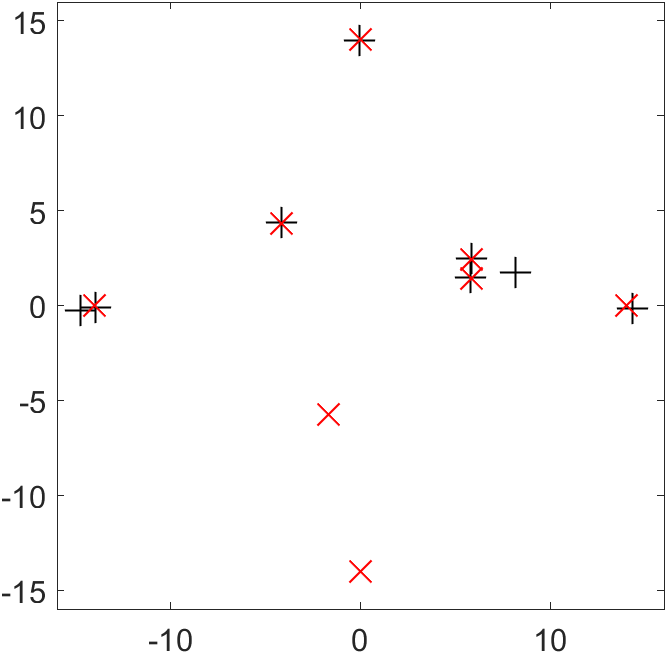}}\quad
	\subfigure[]{\includegraphics[width=0.25\textwidth]{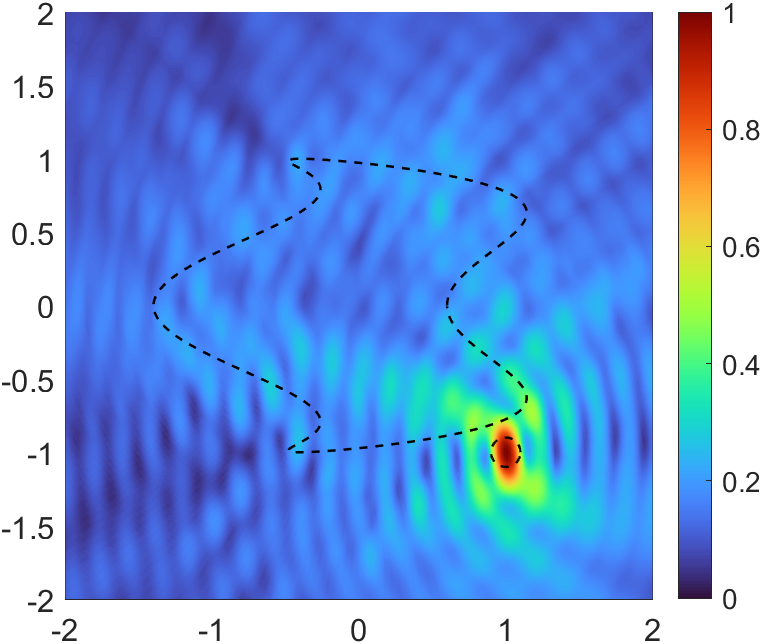}}
	\caption{Reconstruction of the multi-scale obstacles and 8 sources from limited-aperture measurements.}\label{fig: limited}
\end{figure}

\begin{example}[Comparison]
  The last example is devoted to comparing the reconstructions by using different sampling methods. We set the wavenumber $k=14$ and consider a sound-soft starfish-shaped domain whose boundary is parameterized by
  \begin{align*}
  x(t)=(1+0.2\cos5t)(\cos t,\sin t).
  \end{align*}
  In this example, we introduce the different number of auxiliary sources to the co-inversion problem and display the reconstruction in \Cref{fig:starfish}.
  We can see from \Cref{fig:starfish} that the obstacle can be well recognized no matter whether $5$ or $12$ auxiliary source points are added to the geometry setup. When 12 auxiliary sources are added, there is not much difference between the reconstruction results of the two methods, and both two indicator functions recognize the obstacle well. However, our method exhibits better imaging results with higher resolution compared with the conventional method when only 5 auxiliary sources are involved in the co-inversion setup, which shows the feasibility and applicability of our method even though there are not so many auxiliary source points.
\end{example}

\begin{figure} 
	\centering  
	\subfigure[]{\includegraphics[width=0.3\textwidth]{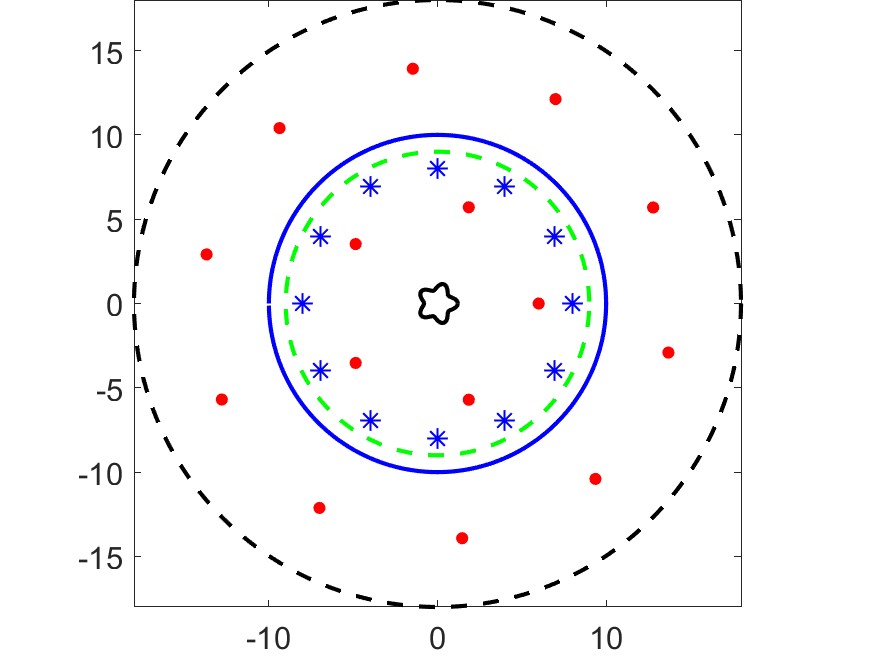}}\quad
	\subfigure[]{\includegraphics[width=0.3\textwidth]{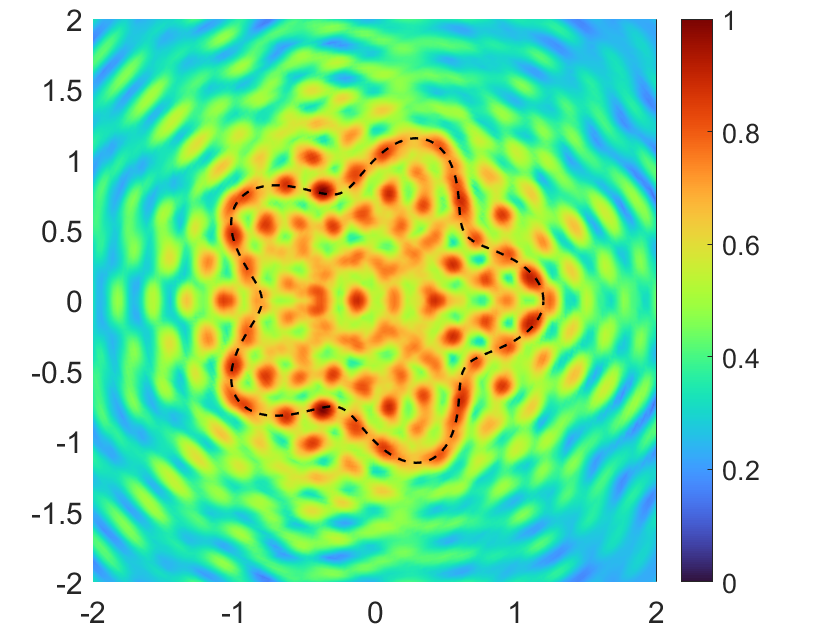}}\quad
	\subfigure[]{\includegraphics[width=0.3\textwidth]{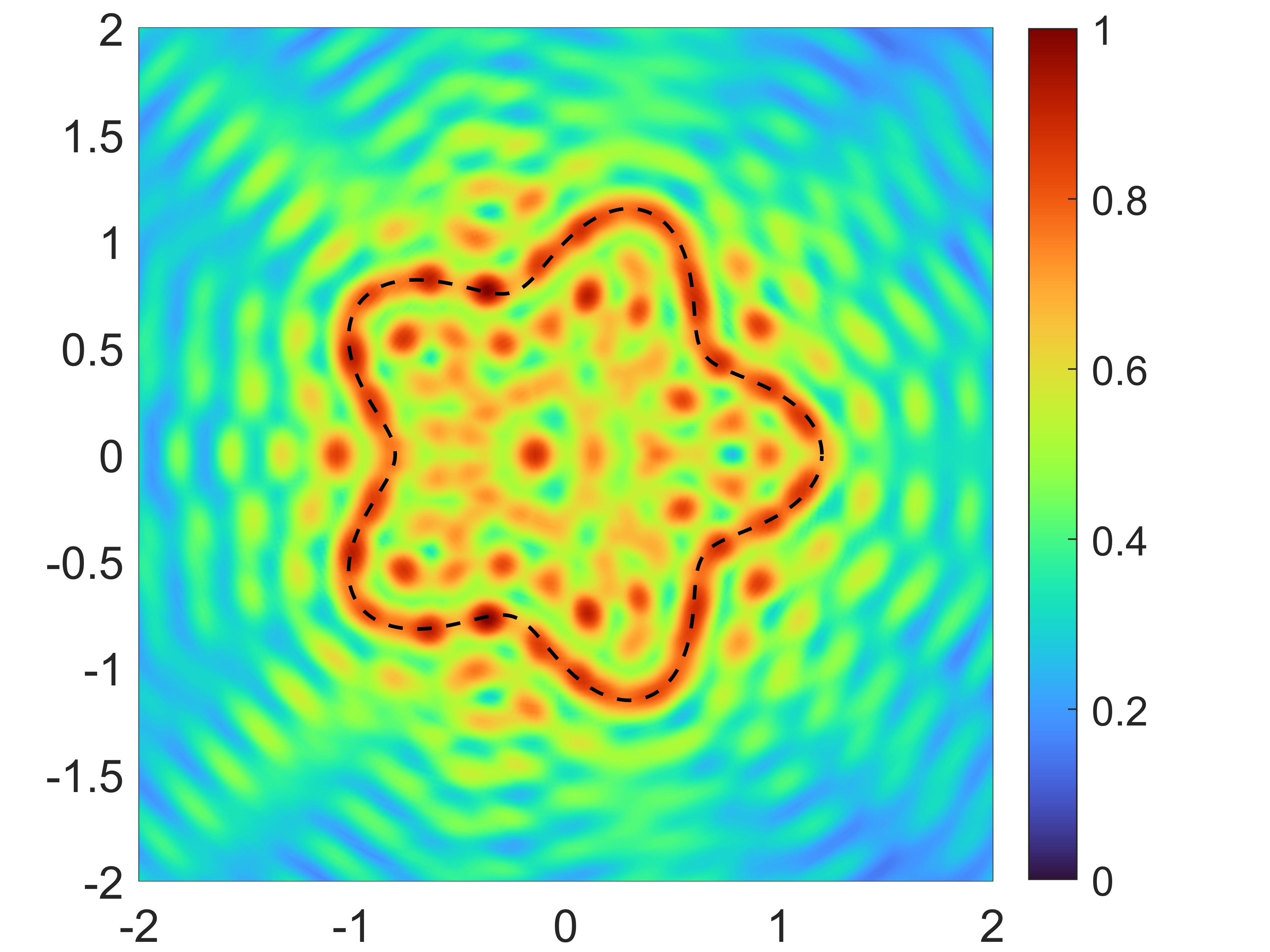}}\quad
	\subfigure[]{\includegraphics[width=0.3\textwidth]{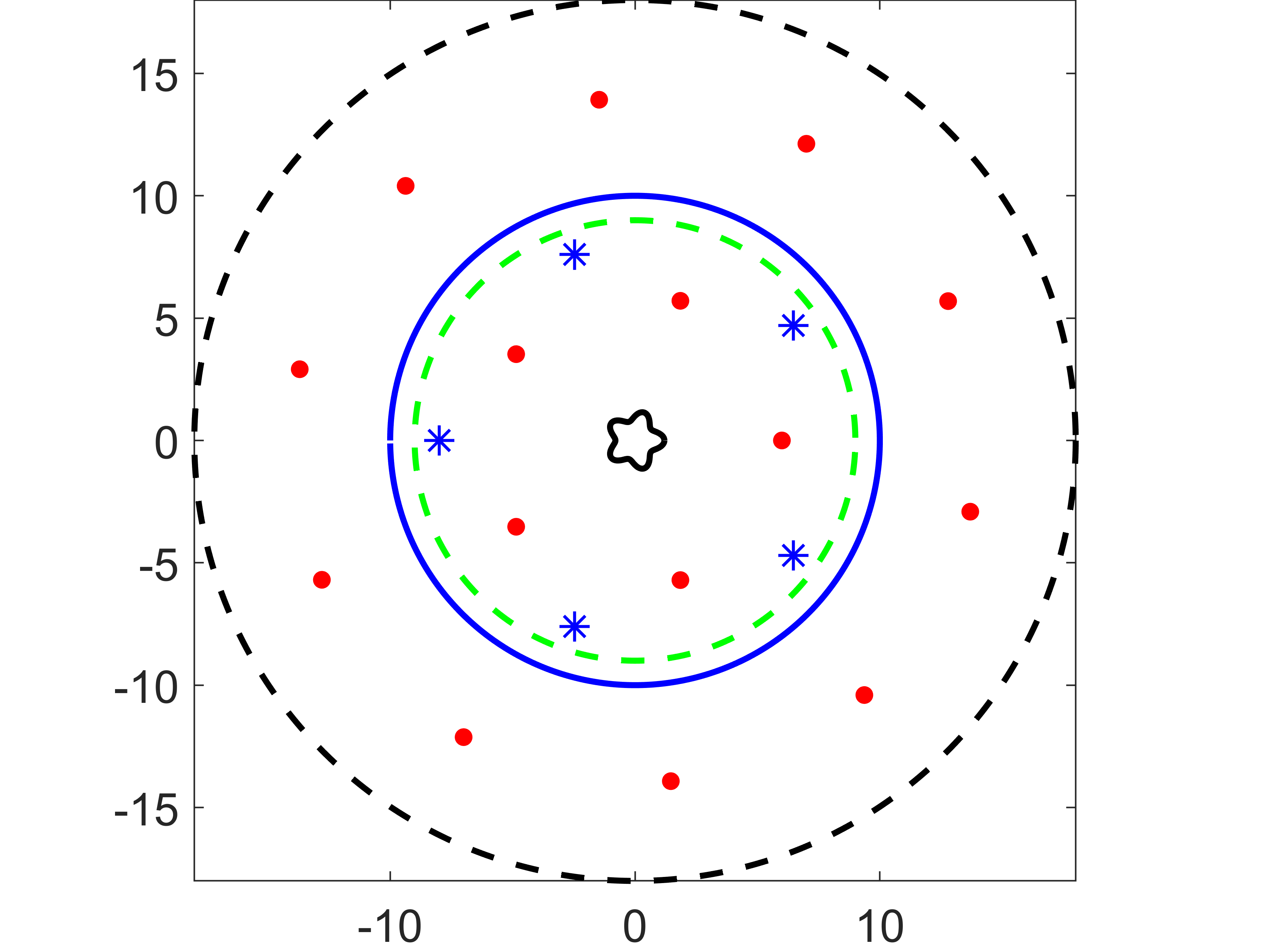}}\quad
	\subfigure[]{\includegraphics[width=0.3\textwidth]{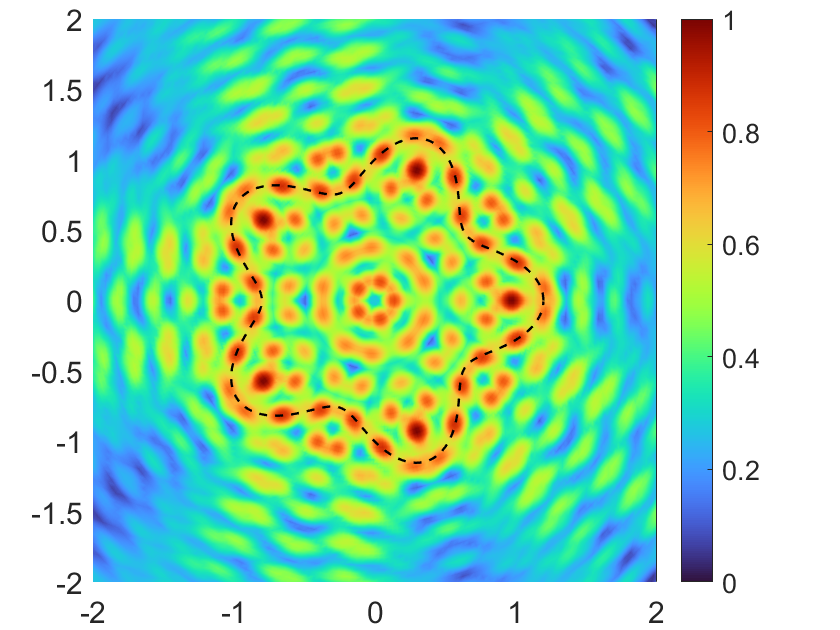}}\quad
	\subfigure[]{\includegraphics[width=0.3\textwidth]{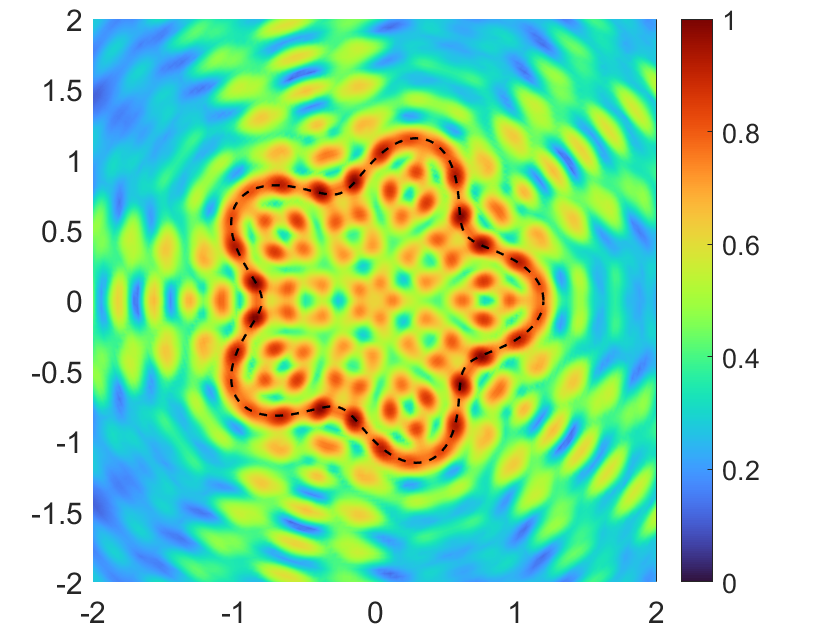}}\quad
	\caption{Reconstruction of the starfish by adding the different number of auxiliary sources to the co-inversion model. The first row: $M=12;$ The second row: $M=5;$ Column 1: Geometry setup; Column 2: Imaging of $I_C;$ Column 2: Imaging of $I_D.$}\label{fig:starfish}
\end{figure}

%%%%%%%%%%%%%%%%%%%%%%%%%%%%%%%%%%%%%%%%%%%%
\section*{Acknowledgments}
%%%%%%%%%%%%%%%%%%%%%%%%%%%%%%%%%%%%%%%%%%%%
	
D. Zhang is supported by NSFC grant 12171200. Y. Guo and Y. Chang are supported by NSFC grant 11971133.
	
%%%%%%%%%%%%%%%%%%%%%%%%%%%%%%%%%%%%%%%%%%%%

\end{document}